\documentclass[preprint,1p]{elsarticle}
\usepackage[pagewise]{lineno}
\usepackage{mathrsfs}
\usepackage{latexsym,amsfonts,amssymb,amsmath,amsthm,color,bbm}
\usepackage{graphicx,subfigure}
\usepackage{amsbsy,amsfonts,amsmath,amssymb,bbm,calc,caption,color,dsfont,graphics}
\usepackage{graphicx,ifthen,mathptmx,mathrsfs}
\usepackage{algorithm}
\usepackage{amsmath}
\usepackage{amsfonts}
\usepackage{epsfig}
\usepackage{epstopdf}
\usepackage[colorlinks=true, linkcolor=blue]{hyperref}
\usepackage{geometry}                  
\begin{document}
	\setlength{\baselineskip}{16pt}

	\newtheorem{theorem}{Theorem}[section]
	\newtheorem{lemma}{Lemma}[section]
	\newtheorem{proposition}{Proposition}[section]
	\newtheorem{definition}{Definition}[section]
	\newtheorem{example}{Example}[section]
	\newtheorem{corollary}{Corollary}[section]
	\newtheorem{assumption}{Assumption}[section]
	\newtheorem{remark}{Remark}[section]
	\numberwithin{equation}{section}
	\renewcommand{\labelenumi}{(\arabic{enumi})}

	\def\disp{\displaystyle}
	\def\undertex#1{$\underline{\hbox{#1}}$}
	\def\card{\mathop{\hbox{card}}}
	\def\sgn{\mathop{\hbox{sgn}}}
	\def\exp{\mathop{\hbox{exp}}}
	\def\OFP{(\Omega,{\cal F},\PP)}
	\newcommand\JM{Mierczy\'nski}
	\newcommand\RR{\ensuremath{\mathbb{R}}}
	\newcommand\EE{\ensuremath{\mathbb{E}}}
	\newcommand\CC{\ensuremath{\mathbb{C}}}
	\newcommand\QQ{\ensuremath{\mathbb{Q}}}
	\newcommand\ZZ{\ensuremath{\mathbb{Z}}}
	\newcommand\NN{\ensuremath{\mathbb{N}}}
	\newcommand\PP{\ensuremath{\mathbb{P}}}
	\newcommand\abs[1]{\ensuremath{\lvert#1\rvert}}
	\newcommand\normf[1]{\ensuremath{\lVert#1\rVert_{f}}}
	\newcommand\normfRb[1]{\ensuremath{\lVert#1\rVert_{f,R_b}}}
	\newcommand\normfRbone[1]{\ensuremath{\lVert#1\rVert_{f, R_{b_1}}}}
	\newcommand\normfRbtwo[1]{\ensuremath{\lVert#1\rVert_{f,R_{b_2}}}}
	\newcommand\normtwo[1]{\ensuremath{\lVert#1\rVert_{2}}}
	\newcommand\norminfty[1]{\ensuremath{\lVert#1\rVert_{\infty}}}
	
	\newcommand\del[1]{}
	
	\begin{frontmatter}

\title{Approximation of invariant measures for random lattice reversible Selkov systems\tnoteref{Supported}}
		\tnotetext[Supported]{The research is supported by National Natural Science Foundation of China (12301020, 12371198), Scientific Research
Program Funds of NUDT (No. 22-ZZCX-016).}

\author{Fang Su}
		\author{Xue Wang\corref{cor1}}
		\author{Xia Pan}

		
		\cortext[cor1]{Corresponding author, wangxue22@nudt.edu.cn}
		
		\address{College of Sciences, National University of Defense Technology, Changsha Hunan, 410073, P.R.China}

\begin{abstract}
This paper focuses on the numerical approximation of random lattice reversible Selkov systems. It establishes the existence of numerical invariant measures for random models with nonlinear noise, using the backward Euler-Maruyama (BEM) scheme for time discretization. The study examines both infinite dimensional discrete random models and their corresponding finite dimensional truncations. A classical path convergence technique is employed to demonstrate the convergence of the invariant measures of the BEM scheme to those of the random lattice reversible Selkov systems. As the discrete time step size approaches zero, the invariant measure of the random lattice reversible Selkov systems can be approximated by the numerical invariant measure of the finite dimensional truncated systems.

\end{abstract}
		

\end{frontmatter}





\section{Introduction}\label{intro}

The spatially discrete representations of partial differential equations are commonly referred to as lattice systems or lattice differential equations. 
For more details on this topic, the reader is directed to Hale's work \cite{Hale-1994}. In recent decades, lattice systems have attracted significant attention from researchers due to their important or potential applications in various fields, including statistical mechanics \cite{Gallavotti-1967}, correlation propagation \cite{Nachtergaele-2006}, image processing \cite{Chua-1993}, pattern recognition \cite{Chow-1995, Chow-1996}, chemical reactions \cite{Kapral-1994}, electrical engineering \cite{Carrol-1990}, and others. Recently, scholars have increasingly recognized the importance of considering the impact of random factors on lattice dynamical systems. When a random partial differential equation is discretized in space, methods and techniques developed for the study of random ordinary differential equations can be effectively adapted to investigate the qualitative behavior of random lattice systems \cite{Arnold-1974, Balenzuela-2014, Boukanjime-2020, Boukanjime-2021, Caraballo-2020, Caraballo-2022, Mao-2006}.

The reversible Selkov system is a significant classical reaction-diffusion equation, commonly used to model autocatalytic biochemical processes. In a broader context, the reversible Selkov model is also known as the two-component Gray-Scott equation. Over the years, several important results have been established regarding the reversible Selkov equations. To date, both deterministic and random forms of these equations have been studied, including random lattice versions. For example, You \cite{You-2012, You-2013} investigated the global attractor and robustness of reversible autocatalytic reaction-diffusion systems. Later, the same author \cite{You-2014} explored random attractors and robustness under stochastic influences. In the case of the random reversible Selkov system driven by multiplicative noise, Guo et al. \cite{Guo-2019} examined the upper semi-continuity of random attractors. The authors in \cite{Gu-2013, Gu-2013-2} introduced pullback and uniform attractors for the non-autonomous three component reversible Gray-Scott system. Furthermore, the existence of random attractors for the random reversible Selkov system on an infinite lattice perturbed by additive noise has been studied in \cite{LiH-2015, LiH-2016, LiH-2019}. More recently, Wang et al. \cite{Wang-2024} analyzed the existence and stability of invariant or periodic probability measures for lattice reversible Selkov systems driven by locally Lipschitz noise. To the best of our knowledge, no study has yet addressed the numerical invariant measure of lattice reversible Selkov systems. Investigating the numerical invariant measure of random reversible Selkov equations is of significant interest, as it is closely related to the convergence analysis of numerical schemes and may provide valuable insights into the upper semicontinuity of numerical invariant measures.

The theory of invariant measures for random differential equations, lattice systems, and partial differential equations has been extensively studied by many authors \cite{Bao-2014, Bessaih-2020, Brzezniak-2016, Brzezniak-2017, Chen-2021, Eckmann-2001, Khasminskii-2012, Wang-2019-0, Wang-2019, Wang-2020, Wangren-2020, Wu-2017} and references therein. In practical applications, understanding the shape of the stationary distribution (or invariant measure) is crucial. For random differential equations, this is often computed by solving the coupled Kolmogorov-Fokker-Planck equations. However, this process is often challenging in practice. Alternatively, numerical schemes can be used to obtain the stationary distribution of random differential equations. When the invariant measure of both the time discrete approximation and the underlying continuous random system is unique, the invariant measures of the discrete approximation will converge to that of the continuous system. For random ordinary differential equations, see \cite{Li-2018, Liu-2015, Shi-2024, Yang-2018, Yuan-2004}, and for random lattice differential systems, see \cite{Caraballo-2024}. For certain complex systems, such as the random 2D Navier-Stokes system, verifying the uniqueness of the invariant measure is challenging (see \cite{Prato-2008, Hairer-2006}). In the absence of uniqueness, the stability of the invariant measures for numerical schemes applied to random differential equations has been recently studied in \cite{Lidingshi-2025, Lidingshi-2026}. In this paper, we investigate the behavior of invariant measures for time discrete approximations in $\ell^2\times \ell^2$. We show that, under certain conditions, any limit point of these invariant measures is an invariant measure of the underlying continuous random lattice system as the step size approaches zero.

In this paper, we consider the following random lattice reversible Selkov system defined on the integer set $mathbb{Z}$:
\begin{equation}\label{1.1}
\left\{\begin{array}{l}
d u_i(t)=\left(d_1\left(u_{i+1}(t)-2 u_i(t)+u_{i-1}(t)\right)-a_1 u_i(t)+b_1 u_i^{2 p}(t) v_i(t)-b_2 u_i^{2 p+1}(t)+f_{i}\right) d t \\
~~~~~~~~~~~~~~~+\left[h_i+\sigma_i(u_i(t))\right] d W(t), \\
d v_i(t)=\left(d_2\left(v_{i+1}(t)-2 v_i(t)+v_{i-1}(t)\right)-a_2 v_i(t)-b_1 u_i^{2 p}(t) v_i(t)+b_2 u_i^{2 p+1}(t)+g_{i}\right) d t \\
~~~~~~~~~~~~~~~+\left[h_i+\sigma_i(v_i(t))\right] d W(t) ,
\end{array}\right.
\end{equation}
with initial conditions
\begin{equation}\label{1.1.0}
u_{i}(\tau)=u_{0, i}, \quad v_{i}(\tau)=v_{0, i},
\end{equation}
where $t>\tau, \tau \in \mathbb{R}, i \in \mathbb{Z}, u=\left(u_{i}\right)_{i \in \mathbb{Z}}, v=\left(v_{i}\right)_{i \in \mathbb{Z}}\in\ell^2, p \geqslant 1$, and $d_1, d_2, a_1, a_2, b_1, b_2$ are positive constants. Additionally, $f_{i}=\left(f_{i}\right)_{i \in \mathbb{Z}}, g_{i}=\left(g_{i}\right)_{i \in \mathbb{Z}} \in\ell^2$ represent the determined external forces. $W$ is a Wiener process defined on a complete filtered probability space $\left(\Omega, \mathcal{F},\left\{\mathcal{F}_t\right\}_{t \in \mathbb{R}}, \mathbb{P}\right)$, and $h=\left(h_{i}\right)_{i \in \mathbb{Z}}, \sigma(\cdot)=\left(\sigma_{i}(\cdot)\right)_{i \in \mathbb{Z}}\in\ell^2$ are the noise intensities.

The purpose of this paper is to demonstrate the existence of numerical invariant measures for the system described by \eqref{1.1}-\eqref{1.1.0}, where the BEM scheme is used to discretize the time direction. We prove that the collection of numerical invariant measures converges upper semicontinuously to the collection of invariant measures for the continuous time random lattice dynamical system as the time step size tends to zero. Furthermore, we consider the finite $N$-dimensional truncations of the BEM scheme to establish the existence of finite dimensional numerical invariant measures and their convergence to the invariant measures of the random lattice dynamical system as $N\rightarrow \infty$. As a result, we obtain upper semicontinuous convergence between the collection of numerical invariant measures for the finite dimensional systems and the collection of invariant measures for the original system \eqref{1.1}-\eqref{1.1.0}. This work is the first to study the numerical approximation of invariant measures for random lattice reversible Selkov systems, building on recent advances in the numerical analysis of dynamical systems.

The paper is organized as follows. In Section 2, we present the necessary hypotheses, apply the BEM scheme to discretize the random lattice system in time, and prove the existence of a unique solution for the BEM scheme. In Section 3, we investigate the existence of the corresponding numerical invariant measures. Section 4 is devoted to proving the convergence of these numerical invariant measures. Finally, in Section 5, we investigate finite dimensional truncations of the discrete random lattice dynamical system and prove that the collection of their numerical invariant measures converges upper semicontinuously to the invariant measures of the original system \eqref{1.1}-\eqref{1.1.0}.

\section{Backward Euler-Maruyama Scheme}\label{sec:2}

Throughout the paper, we will frequently use the following inequalities for all $x,y\in \mathbb{R}$,
\begin{align}\label{2.1}
|x^{r}-y^{r}|\leq C_{r}|x-y||x^{r-1}+y^{r-1}|,~~r\geq1
\end{align}
and
\begin{align}\label{2.2}
b_{1}b_{2}x^{2p+1}y-b_{2}^{2}x^{2p+2}-b_{1}^{2}x^{2p}y^{2}
+b_{2}b_{1}x^{2p+1}y=2b_{1}b_{2}x^{2p+1}y-x^{2p}(b_{2}^{2}x^2 + b_{1}^{2}y^2)\leq 0, ~~p\geq 1.
\end{align}

According to \cite{Gu-2016}, let $\ell^2$
be a Hilbert space of real-valued, square-summable bi-infinite sequences, with the inner product
$$
(u,v)=\sum_{i\in \mathbb{Z}}u_iv_i, \,\,\, \forall \,u=(u_i)_{i\in \mathbb{Z}}, v=(v_i)_{i\in \mathbb{Z}} \in \ell^2,
$$
and the norm $\|u\|=\sqrt{(u, u)}$. That is,
\begin{align*}
\ell^2:=\left\{u=\left(u_i\right)_{i \in \mathbb{Z}}:\|u\|^2=\sum_{i \in \mathbb{Z}}\left|u_i\right|^2<\infty\right\}.
\end{align*}

We introduce the linear operators $A$, $B$ and $B^{*}$ from $\ell^{2}$ to $\ell^{2}$. For any $i\in \mathbb{Z}$ and $u=\left(u_i\right)_{i \in \mathbb{Z}}\in \ell^2$, they are defined as follows:
\begin{align*}
(Au)_{i}=-u_{i-1} + 2u_{i} - u_{i+1}, \quad(B u)_i=u_{i+1}-u_i, \quad\left(B^* u\right)_i:=u_{i-1}-u_i.
\end{align*}
We then have $A=BB^{*}=B^{*}B$, and $(B^{*}u,v)=(u,Bv)$ for all $u,v\in \ell^{2}$. According to \cite{Bates-2006}, all operators are bounded on $\ell^2$ with $\|A\|\leq 4,~\|B\|=\|B^*\|\leq 2$, and $\left(A u, u\right)\geq 0$ for all $u\in \ell^2$. 

Define two operators $F:\ell^{2}\times \ell^{2}\rightarrow \ell^{2}$ and $G:\ell^{2}\times \ell^{2}\rightarrow \ell^{2}$ by $F(u,v)=(u_{i}^{2p}v_{i})_{i\in\mathbb{Z}}$ and $G(u)=(u_{i}^{2p+1})_{i\in\mathbb{Z}}$ for any $u=\{u_{i}\}_{i\in \mathbb{Z}},~v=\{v_{i}\}_{i\in \mathbb{Z}}\in \ell^{2}$. For convenience,  the phase space of the entire paper is denoted by $X=\ell^{2}\times\ell^{2}$.

By \eqref{2.1} and Young's inequality, for any $p \geqslant 1$ and $u_1, v_1, u_2, v_2 \in \ell^2$, it is easy to verify that there exists a constant $C>0$ such that
\begin{equation}\label{2.3}
\left\|F\left(u_1, v_1\right)-F\left(u_2, v_2\right)\right\|^2 \leq C\left(\left\|u_1\right\|_{4 p}^{4 p}+\left\|u_2\right\|_{4 p}^{4 p}+\left\|v_2\right\|_{4 p}^{4 p}\right)\left(\left\|u_1-u_2\right\|^2+\left\|v_1-v_2\right\|^2\right) .
\end{equation}
According to \eqref{2.3}, we conclude that $F(u, v)$ satisfies certain locally Lipschitz conditions. Specifically, for every $n \in \mathbb{N}$, there exists $c_1(n)>0$ such that for any $u_1, u_2, v_1, v_2 \in \ell^2$ with $\left\|u_1\right\| \leqslant n,\left\|u_2\right\| \leqslant n,\left\|v_1\right\| \leqslant n$, and $\left\|v_2\right\| \leqslant n$, the following inequalities hold:
\begin{equation}\label{2.4}
\begin{aligned}
& \left\|F\left(u_1, v_1\right)-F\left(u_2, v_2\right)\right\|^2 \leq c_1(n)\left(\left\|u_1-u_2\right\|^2+\left\|v_1-v_2\right\|^2\right), \\
& \left|\left(F\left(u_1, v_1\right)-F\left(u_2, v_2\right), u_1-u_2\right)\right| \leq c_1(n)\left(\left\|u_1-u_2\right\|^2+\left\|v_1-v_2\right\|^2\right),\\
& \left|\left(F\left(u_1, v_1\right)-F\left(u_2, v_2\right), v_1-v_2\right)\right| \leq c_1(n)\left(\left\|u_1-u_2\right\|^2+\left\|v_1-v_2\right\|^2\right).
\end{aligned}
\end{equation}
In a similar manner, by \eqref{2.1}, we can also verify that $G(u)$ satisfies locally Lipschitz conditions. Specifically, for every $n\in \mathbb{N}$, there exists $c_2(n)>0$ such that for any $u, v, u_1, u_2, v_1, v_2 \in \ell^2$ with $\|u\| \leqslant n,\left\|u_1\right\| \leqslant n,\left\|u_2\right\| \leqslant n,\|v\| \leqslant n,\left\|v_1\right\| \leqslant n$, and $\left\|v_2\right\| \leqslant n$, the following inequalities hold:
\begin{equation}\label{2.5}
\begin{aligned}
& \|G(u)-G(v)\|^2 \leq c_2(n)\|u-v\|^2, \\
& |(G(u_1)-G(u_2), u_1-u_2)| \leq c_2(n)\|u_1-u_2\|^2,\\
& |(G(u_1)-G(u_2), v_1-v_2)| \leq c_2(n)(\|u_1-u_2\|^2+\|v_1-v_2\|^2).
\end{aligned}
\end{equation}

The nonlinear diffusion term $\sigma_{i}=\sigma_{i}(\cdot)$ 
satisfies the Lipschitz continuity condition for any $i\in \mathbb{Z}$
\begin{equation}\label{5-10}
|\sigma_i(s_1) - \sigma_i(s_2)| \leq L_{\sigma} |s_1 - s_2|.
\end{equation}
where $L_{\sigma}>0$ is a constant. Another $\sigma_{i}(\cdot)$ satisfies the linear growth condition for any $i\in \mathbb{Z}$
\begin{equation}\label{5-11}
|\sigma_i(s)| \leq \delta_{i} + \beta|s|,~~~~~~\forall s\in \mathbb{R}.
\end{equation}
Let $\sigma(u)=\left(\sigma_i(u_i) \right)_{i\in \mathbb{Z}}, \delta=\left(\delta_i\right)_{i\in \mathbb{Z}}\in \ell^2$. From \eqref{5-10} and \eqref{5-11}, for any $u,v\in \ell^2$, we have
\begin{equation}\label{5-12}
 \|\sigma(u)\|^2 \leq 2\|\delta\|^2 + 2\beta^{2}\|u\|^2.
\end{equation}
\begin{equation}\label{5-13}
\|\sigma(u) - \sigma(v)\|^2 \leq L_{\sigma}\|u - v\|^2.
\end{equation}

For convenience, we use an abstract system to express the Selkov lattice system \eqref{1.1}-\eqref{1.1.0} in $X$: 
\begin{equation}\label{equ:1}
\left\{\begin{array}{l}
d\psi(t) = (Y(\psi(t)) + M(\psi(t)))dt +  R(\psi(t))dW(t),~~t\geq \tau\in \mathbb{R}, \\
\psi(\tau)=\psi_0=\left(u_0, v_0\right)^{T},
\end{array}\right.
\end{equation}
where $\psi(t)=(u(t), v(t))^T$ as well as $Y(\psi(t))$, $M(\psi(t))$ and $R(\psi(t))$ are defined as
 \begin{equation*}\label{equ:2}
\begin{gathered}
Y(\psi(t))=\left(\begin{array}{c}
-d_1 A u-a_1 u \\
-d_2 A v-a_2 v
\end{array}\right), \quad M(\psi(t))=\left(\begin{array}{c}
b_1 F(u(t), v(t))-b_2 G(u(t)) + f \\
-b_1 F(u(t), v(t))+b_2 G(u(t)) + g
\end{array}\right),\\
R(\psi(t)) = \begin{pmatrix}  h(t) + \sigma(u(t)) \\  h(t) +  \sigma(v(t)) \end{pmatrix}.
\end{gathered}
\end{equation*}

The existence and uniqueness of global-in-time solutions, as well as the existence of the invariant measure for \eqref{equ:1}, are provided in \cite{Wang-2024}. These details are not restated here.

We now give the BEM scheme for the random lattice system \eqref{equ:1}:
\begin{equation}\label{2-1.1}
\left\{\begin{array}{l}
         u_{m+1}^{\Delta} = u_{m}^{\Delta} - d_{1}Au_{m+1}^{\Delta}\Delta - a_{1}u_{m+1}^{\Delta}\Delta + b_{1}F(u_{m+1}^{\Delta},v_{m+1}^{\Delta})\Delta
         -b_{2}G(u_{m+1}^{\Delta})\Delta + f\Delta + \left[h+\sigma( u_{m}^{\Delta})\right]\Delta W_{m},
         \vspace{2.0ex}\\
         v_{m+1}^{\Delta} = v_{m}^{\Delta} - d_{2}Av_{m+1}^{\Delta}\Delta - a_{2}v_{m+1}^{\Delta}\Delta - b_{1}F(u_{m+1}^{\Delta},v_{m+1}^{\Delta})\Delta
         +b_{2}G(u_{m+1}^{\Delta})\Delta + g\Delta + \left[h+\sigma( v_{m}^{\Delta})\right]\Delta W_{m},
 \end{array}\right.
\end{equation}
with initial conditions
\begin{equation*}
	u_{i}(\tau)=u_{0, i}, \quad v_{i}(\tau)=v_{0, i},
\end{equation*}
where $\Delta W_{m} = W((m+1)\Delta) - W(m\Delta)$, $m\in \mathbb{N}_0, \mathbb{N}_0=0\bigcup \mathbb{N}$, and $\Delta > 0$ is the time step.

The existence and uniqueness of the solution for the general random lattice system are established in \cite{Zeidler-1985}. Building on this, we prove the existence and uniqueness of the solution to the BEM scheme \eqref{2-1.1}. We now transform \eqref{2-1.1} into the following form:
\begin{equation}\label{2-1.1-1}
\left\{\begin{array}{l}
         u_{m+1}^{\Delta} + d_{1}Au_{m+1}^{\Delta}\Delta + a_{1}u_{m+1}^{\Delta}\Delta - b_{1}F(u_{m+1}^{\Delta},v_{m+1}^{\Delta})\Delta
         + b_{2}G(u_{m+1}^{\Delta})\Delta  = u_{m}^{\Delta} + f\Delta + \left[h+\sigma( u_{m}^{\Delta})\right]\Delta W_{m},
         \vspace{2.0ex}\\
         v_{m+1}^{\Delta} + d_{2}Av_{m+1}^{\Delta}\Delta + a_{2}v_{m+1}^{\Delta}\Delta + b_{1}F(u_{m+1}^{\Delta},v_{m+1}^{\Delta})\Delta
         - b_{2}G(u_{m+1}^{\Delta})\Delta = v_{m}^{\Delta} + g\Delta + \left[h+\sigma( v_{m}^{\Delta})\right]\Delta W_{m},
 \end{array}\right.
\end{equation}
During the iteration process, the right-hand side is known and denoted by
$D = (D_1, D_2)^{T} \in X$, where
\begin{equation*}
D=\binom{D_1}{D_2}=\binom{u_{m}^{\Delta} + f\Delta + \left[h+\sigma( u_{m}^{\Delta})\right]\Delta W_{m}}{v_{m}^{\Delta} + g\Delta + \left[h+\sigma( v_{m}^{\Delta})\right]\Delta W_{m}}
\end{equation*}
Define the operator $G: X \to X$ as 
\begin{equation*}
G(u,v)=\binom{G_1(u,v)}{G_2(u,v)}=\binom{u + d_{1}Au\Delta + a_{1}u\Delta - b_{1}F(u,v)\Delta + b_{2}G(u)\Delta}{v + d_{2}Av\Delta + a_{2}v\Delta + b_{1}F(u,v)\Delta - b_{2}G(u)\Delta}
\end{equation*}
The equation \eqref{2-1.1-1} is rewritten as
\begin{equation}\label{2-1.1-2}
G(u_{m+1}^{\Delta}, v_{m+1}^{\Delta}) = D
\end{equation}
Let $\psi = (u, v)^T$ (with $u, v \in \ell^2$), and define a weighted inner product space on $X$ by
\begin{equation*}
\langle \psi_1, \psi_2 \rangle = b_2 (u_1, u_2) + b_1 (v_1, v_2)
\end{equation*}
where $b_1, b_2 > 0$, and $(\cdot, \cdot)$ denotes the inner product in $\ell^2$, with the norm given by
\begin{equation*}
\|\psi\|^2_X = \langle \psi, \psi \rangle =  b_2 \|u\|^2 + b_1 \|v\|^2
\end{equation*}

\begin{lemma}\label{lemma2.2} Suppose \eqref{5-10} and \eqref{5-11} hold. When the time step $\Delta>0$ is small enough, there exists a unique solution $\psi\in X$ satisfying \eqref{2-1.1-2}.
\end{lemma}
\begin{proof}
Since
\begin{equation*}
\begin{aligned}
&\langle (u_1-u_2, v_1-v_2)^T, G(u_1,v_1) - G(u_2,v_2) \rangle \\
= &b_2 \|u\|^2 + b_1 \|v\|^2 + b_2\Delta \left( u_1-u_2, d_1 A (u_1-u_2)  + a_1 (u_1-u_2)  - b_1 (u_1^{2p} v_1 - u_2^{2p} v_2)  + b_2 (u_1^{2p+1} - u_2^{2p+1}) \right)\\
& + b_1\Delta \left(v_1-v_2,  d_2 A (v_1-v_2)  + a_2 (v_1-v_2)  + b_1 (u_1^{2p} v_1 - u_2^{2p} v_2)  - b_2 (u_1^{2p+1} - u_2^{2p+1}) \right).
\end{aligned}
\end{equation*}
Since $\left( A (u_1-u_2),(u_1-u_2)\right) \geq 0$, $\left(A (v_1-v_2), (v_1-v_2)\right) \geq 0$, and $b_1 > 0$, $b_2 > 0$, $a_1 > 0$, $a_2 > 0$, and $\Delta > 0$ is small enough, we have
\begin{equation*}
\langle (u_1-u_2, v_1-v_2)^T, G(u_1,v_1) - G(u_2,v_2) \rangle  > 0.
\end{equation*}
Thus, $G$ is strictly monotonic.

Next, we prove that the operator $G$ is coercive. Since
\begin{equation*}
\begin{aligned}
\langle (u,v)^T, G(u,v) \rangle &= b_2 \|u\|^2 + b_1 \|v\|^2 + b_2 d_1 \Delta (u, A u) + b_1 d_2 \Delta (v, A v) + b_2 a_1 \Delta \|u\|^2 + b_1 a_2 \Delta \|v\|^2 \\
&\quad + \Delta \left[ - b_1b_2 (u, u^{2p} v) + b_2^2  (u, u^{2p+1}) + b_1^2 (v, u^{2p} v) - b_1b_2  (v, u^{2p+1}) \right],
\end{aligned}
\end{equation*}
and because $(A u, u) \geq 0$, $(A v, v) \geq 0$, by \eqref{2.2}, we have
\begin{align*}
\langle (u,v)^T, G(u,v)\rangle \geq (1+\lambda\Delta)\left(b_2 \|u\|^2 + b_1 \|v\|^2\right),
\end{align*}
where $\lambda=a_1\bigwedge a_2$. When $\| \psi \|_X = \sqrt{b_2 \|u\|^2 + b_1 \|v\|^2} \rightarrow \infty$, it follows that
\begin{align*}
		\lim_{\| \psi \|_X\rightarrow\infty}\frac{\langle (u,v)^T, G(u,v)\rangle}{\| \psi \|_X} = \infty.
\end{align*}
Thus, the operator $G$ is coercive. By Theorem 26.A in \cite{Zeidler-1985}, there exists a unique solution $\psi\in X$ satisfying \eqref{2-1.1-2}.
\end{proof}

\section{Existence of the numerical invariant measure}

In the following, the existence of an invariant measure for the BEM scheme \eqref{2-1.1} will be considered. Let 
\begin{equation}\label{5-29}
\lambda > 16\beta^{2}.
\end{equation}

\begin{lemma}\label{lemma2.4} Suppose \eqref{5-10}, \eqref{5-11} and \eqref{5-29} hold. Then, for any $0<\Delta<\Delta^*$, the solution of \eqref{2-1.1} satisfies
\begin{equation}
\mathbb{E}\left(\|\psi_{m}^\Delta\|_X^2\right) \leq \|\psi_0\|_X^2 e^{mln\left(1-\frac{\lambda}{4}\Delta\right)} + M, 
\end{equation}
where $M>0$ is independent of $\psi_0$ and $\Delta$.
\end{lemma}
\begin{proof} 
Taking the inner product of \eqref{2-1.1} with $\left(b_2u_{m+1}^{\Delta}, b_1v_{m+1}^\Delta\right)$ in $X$, we get
\begin{equation}\label{2-1.2}
\left\{\begin{array}{l}
b_{2}\|u_{m+1}^{\Delta}\|^2 = - b_{2}d_{1}\Delta\|Bu_{m+1}^{\Delta}\|^2 - b_{2}a_{1}\Delta\|u_{m+1}^{\Delta}\|^2 + b_{2}b_{1}\Delta(F(u_{m+1}^{\Delta},v_{m+1}^{\Delta}),u_{m+1}^{\Delta})\vspace{1.0ex}\\
       \hspace{12.0ex} -b_{2}^2\Delta(G(u_{m+1}^{\Delta}),u_{m+1}^{\Delta}) + b_{2}\Delta(f,u_{m+1}^{\Delta}) + b_{2}\left(u_{m}^{\Delta} +\left[h+\sigma( u_{m}^{\Delta})\right]\Delta W_{m},u_{m+1}^{\Delta}\right),\vspace{2.0ex}\\
b_{1}\|v_{m+1}^{\Delta}\|^2 = - b_{1}d_{2}\Delta\|Bv_{m+1}^{\Delta}\|^2 - b_{1}a_{2}\Delta\|v_{m+1}^{\Delta}\|^2 - b_{1}^2\Delta(F(u_{m+1}^{\Delta},v_{m+1}^{\Delta}),v_{m+1}^{\Delta})\vspace{1.0ex}\\
        \hspace{12.0ex} +b_{1}b_{2}\Delta(G(u_{m+1}^{\Delta}),v_{m+1}^{\Delta}) + b_{1}\Delta(g,v_{m+1}^{\Delta}) + b_{1}\left(v_{m}^{\Delta} + \left[h+\sigma( v_{m}^{\Delta})\right]\Delta W_{m},v_{m+1}^{\Delta})\right).        
\end{array}\right.
\end{equation}
By Young's inequality and \eqref{2.2}, we obtain
\begin{equation}
\begin{aligned}
& b_{2}\|u_{m+1}^{\Delta}\|^2 +  b_{1}\|v_{m+1}^{\Delta}\|^2 \\
\leq & - b_{2}a_{1}\Delta\|u_{m+1}^{\Delta}\|^2 + b_{2}\Delta(f,u_{m+1}^{\Delta}) + b_{2}\left(u_{m}^{\Delta}+\left[h+\sigma( u_{m}^{\Delta})\right]\Delta W_{m},u_{m+1}^{\Delta}\right)\\
& - b_{1}a_{2}\Delta\|v_{m+1}^{\Delta}\|^2 + b_{1}\Delta(g,v_{m+1}^{\Delta}) + b_{1}\left(v_{m}^{\Delta}+\left[h+\sigma( v_{m}^{\Delta})\right]\Delta W_{m},v_{m+1}^{\Delta}\right)\\
\leq & (\frac{1}{2}-\frac{\lambda}{2}\Delta)b_{2}\|u_{m+1}^{\Delta}\|^2 + b_{2}\Delta\frac{1}{2\lambda}\|f\|^2 + \frac{1}{2}b_{2}\|u_{m}^{\Delta}\|^2 +\frac{1}{2}b_{2}\|\left(h+\sigma( u_{m}^{\Delta})\right)\Delta W_{m}\|^2 + b_{2}\left(u_{m}^{\Delta}, \left(h+\sigma( u_{m}^{\Delta})\right)\Delta W_{m}\right)\\
&+ (\frac{1}{2}-\frac{\lambda}{2}\Delta)b_{1}\|v_{m+1}^{\Delta}\|^2 + b_{1}\Delta\frac{1}{2\lambda}\|g\|^2 + \frac{1}{2}b_{1}\|v_{m}^{\Delta}\|^2 +\frac{1}{2}b_{1}\|\left(h+\sigma( v_{m}^{\Delta})\right)\Delta W_{m}\|^2 + b_{1}\left(v_{m}^{\Delta}, \left(h+\sigma( v_{m}^{\Delta})\right)\Delta W_{m}\right).
\end{aligned}
\end{equation}
Taking the expectation, we get
\begin{equation}
\begin{aligned}
&\left(1+\lambda\Delta\right)\mathbb{E}\left(b_{2}\|u_{m+1}^{\Delta}\|^2 +  b_{1}\|v_{m+1}^{\Delta}\|^2\right)\\
\leq & \frac{\Delta}{\lambda}\left(b_{2}\|f\|^2 + b_{1}\|g\|^2\right) + \left(b_{2}\|u_{m}^{\Delta}\|^2 + b_{1}\|v_{m}^{\Delta}\|^2\right) + b_{2}\Delta \left( 2\|h\|^2 + 4\|\delta\|^2 + 4\beta^2\|u_{m}^\Delta\|^2\right) \\
& + b_{1}\Delta \left( 2\|h\|^2 + 4\|\delta\|^2 + 4\beta^2\|v_{m}^\Delta\|^2\right)\\
= & \left(1+4\beta^2\Delta\right)\left(b_{2}\|u_{m}^{\Delta}\|^2 + b_{1}\|v_{m}^{\Delta}\|^2\right) + \Delta \left(\frac{b_{2}}{\lambda}\|f\|^2 + \frac{b_{1}}{\lambda}\|g\|^2 + 2b_2\|h\|^2 +2b_1\|h\|^2 +4b_2\|\delta\|^2 +4b_1\|\delta\|^2  \right).
\end{aligned}
\end{equation}
Then 
\begin{equation}\label{5-31}
\begin{aligned}
&\mathbb{E}\left(b_{2}\|u_{m+1}^{\Delta}\|^2 +  b_{1}\|v_{m+1}^{\Delta}\|^2\right) \\
\leq & \frac{1+4\beta^2\Delta}{1+\lambda\Delta}\left(b_{2}\|u_{m}^{\Delta}\|^2 + b_{1}\|v_{m}^{\Delta}\|^2\right) + \frac{\Delta}{1+\lambda\Delta}\left(\frac{b_{2}}{\lambda}\|f\|^2 + \frac{b_{1}}{\lambda}\|g\|^2 + 2b_2\|h\|^2 +2b_1\|h\|^2 +4b_2\|\delta\|^2 +4b_1\|\delta\|^2  \right).
\end{aligned}
\end{equation}
For any $0<\Delta<\frac{1}{4\lambda}$, we have
\begin{equation}\label{5-32}
\frac{1}{1+\lambda\Delta}\leq 1-\frac{\lambda}{2}\Delta.
\end{equation}
Combining \eqref{5-31} and \eqref{5-32}, for any $\beta^2 < \frac{\lambda}{16}$, we obtain
\begin{equation}\label{5-33}
\begin{aligned}
\mathbb{E}\left(b_{2}\|u_{m+1}^{\Delta}\|^2 +  b_{1}\|v_{m+1}^{\Delta}\|^2\right)\leq &\left(1-\frac{\lambda}{4}\Delta\right)\left(b_2\|u_{m}^\Delta\|^2 + b_1\|v_{m}^\Delta\|^2\right) \\
&+ \Delta\left(\frac{b_{2}}{\lambda}\|f\|^2 + \frac{b_{1}}{\lambda}\|g\|^2 + 2b_2\|h\|^2 +2b_1\|h\|^2 +4b_2\|\delta\|^2 +4b_1\|\delta\|^2  \right).
\end{aligned}
\end{equation}
By induction, we get
\begin{equation}\label{5-34}
\begin{aligned}
&\mathbb{E}\left(\|\psi_{m+1}^\Delta\|_X^2\right) = \mathbb{E}\left(b_{2}\|u_{m+1}^{\Delta}\|^2 +  b_{1}\|v_{m+1}^{\Delta}\|^2\right) \\
\leq & \left(1-\frac{\lambda}{4}\Delta\right)^{m+1}\left(b_2\|u_{0}^\Delta\|^2 + b_1\|v_{0}^\Delta\|^2\right) \\
& + \Delta\left(\frac{b_{2}}{\lambda}\|f\|^2 + \frac{b_{1}}{\lambda}\|g\|^2 + 2b_2\|h\|^2 +2b_1\|h\|^2 +4b_2\|\delta\|^2 +4b_1\|\delta\|^2 \right) \sum^m_{l=0}\left(1-\frac{\lambda}{4}\Delta\right)^{l}\\
\leq & \left(1-\frac{\lambda}{4}\Delta\right)^{m+1}\|\psi_0\|_X^2 + \left(\frac{b_{2}}{\lambda}\|f\|^2 + \frac{b_{1}}{\lambda}\|g\|^2 + 2b_2\|h\|^2 +2b_1\|h\|^2 +4b_2\|\delta\|^2 +4b_1\|\delta\|^2 \right) \frac{4}{\lambda}. 
\end{aligned}
\end{equation}
That is, 
\begin{equation}
\mathbb{E}\left(\|\psi_{m}^\Delta\|_X^2\right) \leq \|\psi_0\|_X^2 e^{mln\left(1-\frac{\lambda}{4}\Delta\right)} + M. 
\end{equation}
\end{proof}

Next, we establish the uniform bounds for the tail ends of the solutions.
\begin{lemma}\label{lemma2.5} Suppose \eqref{5-10}, \eqref{5-11} and \eqref{5-29} hold. Then for any $\eta>0$ and a bounded set $B\subseteq X$, there exists an integer $I=I(\eta,B)\in\mathbb N$, independent of $\Delta$, such that for $0<\Delta<\Delta^*$ and $\psi_{0}\in B$, the solution of \eqref{2-1.1} satisfies
\begin{equation}
\mathbb{E}\left(\sum_{|i|> I}(b_{2}|u_{m,i}^{\Delta}|^{2}+b_{1}|v_{m,i}^{\Delta}|^{2})\right)
\leq \left(\sum_{|i|> I}(b_{2}|u_{0,i}^{\Delta}|^{2}+b_{1}|v_{0,i}^{\Delta}|^{2})\right) e^{mln\left(1-\frac{\lambda}{4}\Delta\right)} + \eta. 
\end{equation}
\end{lemma}
\begin{proof}
we now introduce the cutoff function $\theta:\mathbb R^{+}\rightarrow [0,1]$, which is continuously differentiable and satisfies the following conditions:
\begin{equation*}\theta(s)=
\left\{\begin{array}{l}
0,~~0\leq s\leq1,\\
1,~~s\geq 2.
\end{array}\right.
\end{equation*}
Note that $\theta'$ is bounded on $R$, i.e., there exists a constant $c_0$ such that $|\theta'(s)| \leq c_0$ for $s\in R$. 

For any $n\in\mathbb N$, let $\theta_n = \left(\theta\left(\frac{i}{n}\right)\right)_{i\in\mathbb Z}$, and define $\theta_{n}u_m^\Delta = \left(\theta\left(\frac{i}{n}\right)  u_{m,i}^\Delta  \right)_{i\in\mathbb Z}$. From the BEM scheme \eqref{2-1.1}, we have
\begin{equation}\label{5-35}
\left\{\begin{array}{l}
         \theta_{n}u_{m+1}^\Delta = - d_{1}\theta_{n}Au_{m+1}^{\Delta}\Delta - a_{1}\theta_{n}u_{m+1}^{\Delta}\Delta + b_{1}\theta_{n}F(u_{m+1}^{\Delta},v_{m+1}^{\Delta})\Delta
         -b_{2}\theta_{n}G(u_{m+1}^{\Delta})\Delta \vspace{1.0ex}\\
        \hspace{9.0ex} + \theta_{n}f\Delta + \left(\theta_{n}u_{m}^\Delta  + \left[\theta_{n}h+\theta_{n}\sigma( u_{m}^{\Delta})\right]\Delta W_{m}\right),
         \vspace{2.0ex}\\
         \theta_{n}v_{m+1}^{\Delta} = - d_{2}\theta_{n}Av_{m+1}^{\Delta}\Delta - a_{2}\theta_{n}v_{m+1}^{\Delta}\Delta - b_{1}\theta_{n}F(u_{m+1}^{\Delta},v_{m+1}^{\Delta})\Delta
         +b_{2}\theta_{n}G(u_{m+1}^{\Delta})\Delta \vspace{1.0ex}\\
         \hspace{9.0ex} + \theta_{n}g\Delta +  \left(\theta_{n}v_{m}^{\Delta} + \left[\theta_{n}h+\theta_{n}\sigma( v_{m}^{\Delta})\right]\Delta W_{m}\right).
 \end{array}\right.
\end{equation}
Taking the inner product of \eqref{5-35} with $\left(b_2\theta_{n}u_{m+1}^{\Delta}, b_1\theta_{n}v_{m+1}^\Delta\right)$ in $X$, we get
\begin{equation}\label{5-36}
\begin{aligned}
b_{2}\|\theta_{n}u_{m+1}^{\Delta}\|^2 =& \left( b_{2}\theta_{n}u_{m+1}^\Delta,  - d_{1}\theta_{n}Au_{m+1}^{\Delta}\Delta - a_{1}\theta_{n}u_{m+1}^{\Delta}\Delta + b_{1}\theta_{n}F(u_{m+1}^{\Delta},v_{m+1}^{\Delta})\Delta - b_{2}\theta_{n}G(u_{m+1}^{\Delta})\Delta + \theta_{n}f\Delta \right) \vspace{1.0ex}\\
\hspace{9.0ex} & + \left(b_{2}\theta_{n}u_{m+1}^\Delta, \theta_{n}u_{m}^\Delta + \left[\theta_{n}h+\theta_{n}\sigma( u_{m}^{\Delta})\right]\Delta W_{m} \right)\\
=& - b_{2} d_{1}\Delta\left(Bu_{m+1}^\Delta,B(\theta_{n}^2u_{m+1}^{\Delta})\right) - b_2a_{1}\Delta\|\theta_{n}u_{m+1}^{\Delta}\|^2 + b_{2}b_{1}\Delta\left(\theta_{n}u_{m+1}^{\Delta}, \theta_{n}F(u_{m+1}^{\Delta},v_{m+1}^{\Delta})\right) \vspace{1.0ex}\\
\hspace{9.0ex} & - b_{2}^2\Delta\left(\theta_{n}u_{m+1}^{\Delta}, \theta_{n}G(u_{m+1}^{\Delta}) \right) + b_{2}\Delta \left( \theta_{n}u_{m+1}^{\Delta}, \theta_{n}f \right) +  b_{2}\left( \theta_{n}u_{m+1}^\Delta, \theta_{n}u_{m}^\Delta + \left[\theta_{n}h+\theta_{n}\sigma( u_{m}^{\Delta})\right]\Delta W_{m} \right),
\end{aligned}
\end{equation}
\begin{equation}\label{5-36-1}
\begin{aligned}
b_{1}\|\theta_{n}v_{m+1}^{\Delta}\|^2 = &\left( b_{1}\theta_{n}v_{m+1}^\Delta, - d_{2}\theta_{n}Av_{m+1}^{\Delta}\Delta - a_{2}\theta_{n}v_{m+1}^{\Delta}\Delta - b_{1}\theta_{n}F(u_{m+1}^{\Delta},v_{m+1}^{\Delta})\Delta + b_{2}\theta_{n}G(u_{m+1}^{\Delta})\Delta + \theta_{n}g\Delta \right) \vspace{1.0ex}\\
\hspace{9.0ex} & + \left(b_1\theta_{n}v_{m+1}^\Delta, \theta_{n}v_{m}^\Delta + \left[\theta_{n}h+\theta_{n}\sigma( v_{m}^{\Delta})\right]\Delta W_{m} \right)\\
=& - b_{1} d_{2}\Delta\left(Bv_{m+1}^\Delta,B(\theta_{n}^2v_{m+1}^{\Delta})\right) - b_{1}a_{2}\Delta\|\theta_{n}v_{m+1}^{\Delta}\|^2 - b_{1}^2\Delta\left(\theta_{n}v_{m+1}^{\Delta}, \theta_{n}F(u_{m+1}^{\Delta},v_{m+1}^{\Delta})\right) \vspace{1.0ex}\\
\hspace{9.0ex} & + b_{1}b_{2}\Delta\left(\theta_{n}v_{m+1}^{\Delta}, \theta_{n}G(u_{m+1}^{\Delta}) \right) + b_{1}\Delta \left( \theta_{n}v_{m+1}^{\Delta}, \theta_{n}g \right) +  b_{1}\left( \theta_{n}v_{m+1}^\Delta, \theta_{n}v_{m}^\Delta + \left[\theta_{n}h+\theta_{n}\sigma( v_{m}^{\Delta})\right]\Delta W_{m} \right).     
\end{aligned}
\end{equation}
For the first term on the right hand side of \eqref{5-36}, we get
\begin{equation}\label{5-36-1}
\begin{aligned}
& - b_{2}d_{1}\Delta\left(Bu_{m+1}^\Delta,B(\theta_{n}^2u_{m+1}^{\Delta})\right)\\
=& - b_{2}d_{1}\Delta \sum_{i \in \mathbb{Z}}\left[\left(u_{m+1,i+1}^\Delta - u_{m+1,i}^\Delta\right)\left( \theta^2\left(\frac{i+1}{n}\right) u_{m+1,i+1}^\Delta - \theta^2\left(\frac{i}{n}\right) u_{m+1,i}^\Delta\right)\right]\\
=& - b_{2}d_{1}\Delta \sum_{i \in \mathbb{Z}}\left[ \theta^2\left(\frac{i+1}{n}\right) \left(u_{m+1,i+1}^\Delta - u_{m+1,i}^\Delta\right)^2\right] - 
 b_{2}d_{1}\Delta \sum_{i \in \mathbb{Z}}\left[ \left(u_{m+1,i+1}^\Delta - u_{m+1,i}^\Delta\right) \left( \theta^2\left(\frac{i+1}{n}\right) -  \theta^2\left(\frac{i}{n}\right)\right) u_{m+1,i}^\Delta \right]\\
\leq & 2b_{2}d_{1}\Delta \sum_{i \in \mathbb{Z}}|\theta\left(\frac{i+1}{n}\right) -  \theta\left(\frac{i}{n}\right)||u_{m+1,i+1}^\Delta - u_{m+1,i}^\Delta||u_{m+1,i}^\Delta|\\
\leq & \Delta\frac{c(\theta)}{n} \|u_{m+1}^\Delta\|^2.
\end{aligned}
\end{equation}
By Lemma \ref{lemma2.4}, we obtain 
\begin{equation}\label{5-37}
 - b_{2}d_{1}\Delta\mathbb{E}\left(Bu_{m+1}^\Delta,B(\theta_{n}^2u_{m+1}^{\Delta})\right) \leq  \Delta\frac{c_1(\theta)}{n} \left( 1 + \|u_{0}\|^2 \right).
\end{equation}
Similarly, we can easily derive that
\begin{equation}\label{5-37-1}
 - b_{1}d_{2}\Delta\mathbb{E}\left(Bv_{m+1}^\Delta,B(\theta_{n}^2v_{m+1}^{\Delta})\right) \leq  \Delta\frac{c_2(\theta)}{n} \left( 1 + \|v_{0}\|^2 \right).
\end{equation}
By Young's inequality, it follows that
\begin{equation}\label{5-38}
\begin{aligned}
& b_{2}\Delta\left(\theta_n u_{m+1}^\Delta,\theta_{n}f\right) \leq b_2\Delta\frac{\lambda}{2} \|\theta_n u_{m+1}^\Delta\|^2 + b_2\Delta\frac{1}{2\lambda} \sum_{|i|>n} |f_{i}|^2,\\
& b_{2}\left( \theta_n u_{m+1}^\Delta,\theta_{n}u_{m}^\Delta + \left[\theta_{n}h+\theta_{n}\sigma( u_{m}^{\Delta})\right]\Delta W_{m} \right)
\leq  \frac{b_2}{2}\|\theta_n u_{m+1}^\Delta\|^2 + \frac{b_2}{2} \| \theta_{n}u_{m}^\Delta + \left[\theta_{n}h+\theta_{n}\sigma( u_{m}^{\Delta})\right]\Delta W_{m}\|^2,
\end{aligned}
\end{equation}
\begin{equation}\label{5-38-1}
\begin{aligned}
& b_{1}\Delta\left(\theta_n v_{m+1}^\Delta,\theta_{n}g\right) \leq b_1\Delta\frac{\lambda}{2} \|\theta_n v_{m+1}^\Delta\|^2 + b_1\Delta\frac{1}{2\lambda} \sum_{|i|>n} |g_{i}|^2,\\
& b_{1}\left( \theta_n v_{m+1}^\Delta,\theta_{n}v_{m}^\Delta + \left[\theta_{n}h+\theta_{n}\sigma( v_{m}^{\Delta})\right]\Delta W_{m} \right)
\leq  \frac{b_1}{2}\|\theta_n v_{m+1}^\Delta\|^2 + \frac{b_1}{2} \| \theta_{n}v_{m}^\Delta + \left[\theta_{n}h+\theta_{n}\sigma( v_{m}^{\Delta})\right]\Delta W_{m}\|^2.
\end{aligned}
\end{equation}
Combining \eqref{5-36}-\eqref{5-38-1} and \eqref{2.2}, we obtain
\begin{equation}\label{5-40}
\begin{aligned}
b_{2}\|\theta_{n}u_{m+1}^{\Delta}\|^2 \leq & b_{2}\left(\frac{1}{2} - \frac{\lambda}{2}\Delta \right) \|\theta_{n}u_{m+1}^\Delta\|^2 + \left( \frac{c(\theta)}{n} \|u_{m+1}^\Delta\|^2 + \frac{b_2}{2\lambda} \sum_{|i|>n} |f_{i}|^2 \right)\Delta \\
& + \frac{b_2}{2} \| \theta_{n}u_{m}^\Delta + \left[\theta_{n}h + \theta_{n}\sigma( u_{m}^{\Delta})\right]\Delta W_{m}\|^2,
\end{aligned}
\end{equation}
\begin{equation}\label{5-40-1}
\begin{aligned}
b_{1}\|\theta_{n}v_{m+1}^{\Delta}\|^2 \leq & b_{1}\left(\frac{1}{2} - \frac{\lambda}{2}\Delta \right) \|\theta_{n}v_{m+1}^\Delta\|^2 + \left( \frac{\tilde{c}(\theta)}{n} \|v_{m+1}^\Delta\|^2 + \frac{b_1}{2\lambda} \sum_{|i|>n} |g_{i}|^2 \right)\Delta \\
& + \frac{b_1}{2} \| \theta_{n}v_{m}^\Delta + \left[\theta_{n}h + \theta_{n}\sigma( v_{m}^{\Delta})\right]\Delta W_{m}\|^2.
\end{aligned}
\end{equation}
Using a proof similar to that of inequality \eqref{5-34}, we obtain
\begin{equation}\label{5-41}
\begin{aligned}
\mathbb{E}\left[ b_{2}\|\theta_{n}u_{m}^{\Delta}\|^2 + b_{1}\|\theta_{n}v_{m}^{\Delta}\|^2 \right] \leq 
& \left( b_{2}\|\theta_{n}u_{0}\|^2 + b_{1}\|\theta_{n}v_{0}\|^2 \right)e^{mln\left(1-\frac{\lambda}{4}\Delta\right)}\\
& + 2 \left[ \frac{c_1(\theta)}{n} (1+\|u_{0}\|^2) + \frac{c_2(\theta)}{n} (1+\|v_{0}\|^2) + \frac{b_2}{2\lambda} \sum_{|i|>n} |f_{i}|^2  + \frac{b_1}{2\lambda} \sum_{|i|>n} |g_{i}|^2 \right. \\
& \left. + (b_{2} + b_{1})\sum_{|i|>n} |h_{i}|^2 + 2(b_{2} + b_{1})\sum_{|i|>n} |\delta_{i}|^2 \right] \Delta \sum_{l=0}^{m-1}(1-\frac{\lambda}{4}\Delta)^l.
\end{aligned}
\end{equation}
For any $\eta_1 >0$, there exists  $N_1=N_1(\eta_1, B)\in \mathbb{N}$ such that for all $n \geq N_1$, it follows that
\begin{equation}\label{5-42}
2 \left( \frac{c_1(\theta)}{n} (1+\|u_{0}\|^2) + \frac{c_2(\theta)}{n} (1+\|v_{0}\|^2) \right) < \frac{\eta_1}{2}.
\end{equation}
Since $f=\left(f_{i}\right)_{i \in \mathbb{Z}}, g=\left(g_{i}\right)_{i \in \mathbb{Z}}, h=\left(h_{i}\right)_{i \in \mathbb{Z}}, \delta=\left(\delta_{i}\right)_{i \in \mathbb{Z}} \in\ell^2$, we find that there exists $N_2=N_2(\eta_1)\geq N_1$ such that for all $n\geq N_2$,
\begin{equation}\label{5-43}
2\left(\frac{b_2}{2\lambda} \sum_{|i|>n} |f_{i}|^2  + \frac{b_1}{2\lambda} \sum_{|i|>n} |g_{i}|^2 + (b_{2} + b_{1})\sum_{|i|>n} |h_{i}|^2 + 2(b_{2} + b_{1})\sum_{|i|>n} |\delta_{i}|^2 \right) <\frac{\eta_1}{2}.
\end{equation}
From \eqref{5-40}-\eqref{5-43}, we have
\begin{equation}\label{5-41}
\begin{aligned}
\mathbb{E}\left[ b_{2}\|\theta_{n}u_{m}^{\Delta}\|^2 + b_{1}\|\theta_{n}v_{m}^{\Delta}\|^2 \right] \leq & \left( b_{2}\|\theta_{n}u_{0}\|^2 + b_{1}\|\theta_{n}v_{0}\|^2 \right)e^{mln\left(1-\frac{\lambda}{4}\Delta\right)} + \eta_1\Delta \sum_{l=0}^{m-1}\left(1-\frac{\lambda}{4}\Delta\right)^l\\
\leq & \left( b_{2}\|\theta_{n}u_{0}\|^2 + b_{1}\|\theta_{n}v_{0}\|^2 \right)e^{mln\left(1-\frac{\lambda}{4}\Delta\right)} + \eta.
\end{aligned}
\end{equation}
The proof is complete.
\end{proof}

According to Theorem 2.7 in \cite{Liu-2015}, the sequence $\{\psi_m^\Delta\}_{m\in \mathbb{N}}$ is a time homogeneous Markov process with the Feller property. We now establish the tightness of the family of probability distributions corresponding to the solutions of \eqref{2-1.1}.

\begin{lemma}\label{lemma2.6} Suppose \eqref{5-10}, \eqref{5-11} and \eqref{5-29} hold. Then the distribution laws of the discrete process $\{\psi_m^\Delta\}_{m\in \mathbb{N}}$ are tight on $X$.
\end{lemma}
\begin{proof}
Using Lemmas \ref{lemma2.4} and \ref{lemma2.5}, the tightness of the distributions for $\{\psi_m^\Delta\}_{m\in \mathbb{N}}$ on $X$ follows from the results presented in \cite{Wang-2019}. 
The details are similar and therefore omitted here. 
\end{proof}

Subsequently, we demonstrate the existence of invariant measures for the BEM scheme \eqref{2-1.1}.
\begin{theorem}\label{theorem2.1} Suppose \eqref{5-10}, \eqref{5-11} and \eqref{5-29} hold. Then the BEM scheme \eqref{2-1.1} has a numerical invariant measure on $X$.
\end{theorem}
\begin{proof} Let $P^\Delta(t_j, \psi_0^\Delta; t_i, \Gamma)$ denote the transition probability operators for a time homogeneous Markov process $\psi_m^\Delta$ of BEM scheme \eqref{2-1.1}, where $\psi_0^\Delta \in X$, $0 \leq t_j \leq t_i$ with $i, j \in \mathbb{N}_0$, and $\Gamma \in B(X)$. For any $z \in X, i, n \in \mathbb{N}$, define
\begin{equation}\label{Th3.1-1}
\mu_n^\Delta = \frac{1}{n} \sum_{i=0}^n P^\Delta(0, z; t_i, \cdot), 
\end{equation}
where $t_i = i\Delta$. According to Lemma \ref{lemma2.6}, the sequence $\{\mu_n^\Delta\}_{n=1}^\infty$ is tight, and there exists a probability measure $\mu^\Delta$ on $X$ such that
\begin{equation}\label{Th3.1-2}
\mu_n^\Delta \to \mu^\Delta, \text{ as } n \to \infty. 
\end{equation}
Using the Krylov-Bogolyubov method and by \eqref{Th3.1-1}-\eqref{Th3.1-2} (See \cite{Lidingshi-2026} Theorem 3.7), for any $i\in \mathbb{N}_0$, we have
\begin{equation*}
\int_{X} \left( \int_{X} \phi(x) P^\Delta(0, y; t_i, dx) \right) d\mu^\Delta(y) = \int_{X} \phi(y) d\mu^\Delta(y),
\end{equation*}
which shows that $\mu^\Delta$ is the numerical invariant measure for the BEM scheme \eqref{2-1.1}.
\end{proof}

\section{Upper semi-continuous of the numerical invariant measure}\label{sec:3}

In this section, we demonstrate that the collection of invariant measures, denoted $S^\Delta$, for the BEM scheme \eqref{2-1.1}, converges upper semi-continuously to the collection $S^0$ of invariant measures for the continuous system \eqref{equ:1} as the time step $\Delta$ approaches zero. From Theorems \ref{theorem2.1} and \cite{Wang-2024}, we know that both $S^\Delta$ and $S^0$ are nonempty. First, we present an abstract result, established in \cite{Lidingshi-2025}, which states that uniform convergence in probability for the time discrete approximation, along with the uniform continuity in probability of the underlying continuous system, guarantees the convergence of the invariant measures.

Let $(\mathrm{Z}, d)$ be a Polish space equipped with the metric $d$. For each $\Delta \in (0, 1]$ and $x \in \mathrm{Z}$, assume that $\{X_k^\Delta(x), k \in \mathbb{N}\}$ forms a stochastic sequence on the probability space $(\Omega, \mathcal{F}, \{\mathcal{F}_t\}_{t \in \mathbb{R}^+}, P)$, starting with $x$ at time $0$. Additionally, $\{X^0(t, 0, x), t \geq 0\}$ represents a stochastic process on $(\Omega, \mathcal{F}, \{\mathcal{F}_t\}_{t \in \mathbb{R}^+}, P)$ with the initial condition $x(0) = x$. The probability transition operators associated with $X^0(t, 0, x)$ are Feller.

For any compact set $K\subseteq \mathrm{Z}$, $t\geq 0$, and $\eta > 0$, we define the following conditions:
\begin{equation}\label{5-52}
\lim_{\Delta\rightarrow 0}\sup_{x\in K}\mathbf{P}\left( d\left(X^{\Delta}_{\left[\frac{t}{\Delta}\right]}\left( x \right), X^{0}\left(\left[\frac{t}{\Delta}\right]\Delta, 0, x \right)\right) \geq \eta\right) = 0,
\end{equation}
\begin{equation}\label{5-53}
\lim_{\Delta\rightarrow 0}\sup_{x\in K}\mathbf{P}\left( d\left(X^{0}\left(\left[\frac{t}{\Delta}\right]\Delta, 0, x \right), X^{0}\left(t, 0, x \right)\right) \geq \eta\right) = 0.
\end{equation}
where $[x]$ is the greatest integer less than or equal to $x\in R$.

\begin{theorem}\label{theorem3.1}(\cite{Lidingshi-2025}) Suppose \eqref{5-52}, \eqref{5-53} hold and $\Delta_{n}\rightarrow 0$. If $\mu$ is a probability measure on $Z$, and $\mu^{\Delta_{n}}$ is an invariant measure of $X^{\Delta_{n}}$, with $\mu^{\Delta_n}$ weakly converging to $\mu$, then $\mu$ must be an invariant measure of $X^0$.
\end{theorem}

Define $C_{b}(X)$ as the space of all bounded continuous functions $\varphi: X\rightarrow R$, with the norm
\begin{equation*}
\|\varphi\|_{\infty} = \sup_{x\in X}|\varphi(x)|.
\end{equation*}
Let $L_{b}(X)$ represent the space of bounded Lipschitz functions on $X$, consisting of all functions $\varphi \in C_{b}(X)$ that satisfy the condition:
\begin{equation*}
Lip(\varphi):= \sup_{x_{1},x_{2}\in X,~ x_{1}\neq x_{2}} \frac{|\varphi(x_{1}) - \varphi(x_{2})|}{\|x_{1} - x_{2}\|} < \infty.
\end{equation*}
The space $L_{b}(X)$ is equipped with the norm given by
\begin{equation*}
\|\varphi\|_{L} = \|\varphi\|_{\infty} + Lip(\varphi).
\end{equation*}
Let $P(X)$ denote the set of probability measures on $\left(X,B(X)\right)$, with $B(X)$ being the Borel $\sigma$-algebra on $X$. For any $\varphi\in C_{b}(X)$ and $\mu\in P(X)$, define     
\begin{equation*}
(\varphi,\mu) = \int_{X}\varphi(x)\mu(dx).
\end{equation*}
Note that a sequence $\{\mu_{n}\}_{n=1}^{\infty}\subseteq P(X)$ converges weakly to $\mu \in P(X)$. For any $\varphi\in C_{b}(X)$, we have
\begin{equation*}
\lim_{n\rightarrow \infty}(\varphi,\mu_n) = (\varphi,\mu).
\end{equation*}
Define a metric by
\begin{equation*}
d_{P(X)}(\mu_{1}, \mu_{2}) = \sup_{\varphi\in L_{b}(X),~\|\varphi\|_{L}\leq 1} |(\varphi,\mu_{1})-(\varphi,\mu_{2})|,~~~~~~~\forall \mu_{1},~\mu_{2}\in P(X).
\end{equation*}
Thus $\left( P(X),~ d_{P(X)} \right)$ forms a Polish space. Furthermore, a sequence $\{\mu_n\}_{n=1}^\infty\subseteq P(X)$ converges to $\mu$ in $\left( P(X),~ d_{P(X)} \right)$ if and only if $\{\mu_n\}_{n=1}^\infty$ converges weakly to $\mu$.

The main result of this section is stated below.
\begin{theorem}\label{theorem3.2} Suppose \eqref{5-10}, \eqref{5-11} and \eqref{5-29} hold. Then
\begin{equation}\label{5-60}
\lim_{\Delta\rightarrow 0} d_{P(X)} \left(S^\Delta,S^0\right) = 0.
\end{equation}
\end{theorem}
\begin{proof}
(i) Firstly, we demonstrate that the set $\cup_{\Delta\in(0,\Delta^{*})}S^\Delta$ is tight in the space $\left( P(X),~ d_{P(X)} \right)$.
According to Lemma \ref{lemma2.4}, for any $\xi\in X$, there exist constants $T_1>0$ and $c_1>0$ independent of $\Delta$, such that 
\begin{equation*}
\mathbb{E}\left(\|\psi_m^\Delta\|_X^2\right)\leq c_1,  ~~~~for~ all~  m \geq \left[ \frac{T_1}{-\ln\left(1-\frac{\lambda}{4}\Delta\right)}\right]+1,
\end{equation*}
where $\psi_m^\Delta(\xi)$ is abbreviated as $\psi_m^\Delta$. Using Chebyshev's inequality, for any $\eta>0$, there exists $R_1=R_1(\eta)>0$, independent of $\triangle$, such that for all $m \geq \left[ \frac{T_1}{-\ln\left(1-\frac{\lambda}{4}\Delta\right)}\right]+1$, we have
\begin{equation}\label{5-61}
\mathbf{P} \left( \|\psi_m^\Delta\|_X^2  \geq R_1 \right) \leq \frac{\eta}{2}.
\end{equation}
By Lemma \ref{lemma2.5}, for any $\psi_0\in X,~\eta>0$ and $k\in\mathbb N$, there exists an integer $n_k=n_k(\eta,k)$ independent of $\Delta$, such that for all $T_2>T_1$ and $m \geq \left[ \frac{T_2}{-\ln\left(1-\frac{\lambda}{4}\Delta\right)}\right]+1$, the following holds:
\begin{equation*}
\mathbb{E}\left(\sum_{|i| > n_k}(b_{2}|u_{m,i}^{\Delta}|^{2}+b_{1}|v_{m,i}^{\Delta}|^{2})\right)< \frac{\eta}{2^{2k}}.
\end{equation*}
Therefore, for all $m \geq \left[ \frac{T_2}{-\ln\left(1-\frac{\lambda}{4}\Delta\right)}\right]+1$ and $k\in \mathbb{N}$, we obtain
\begin{equation}\label{5-62}
\mathbf{P} \left(\sum_{|i|> n_k}(b_{2}|u_{m,i}^{\Delta}|^{2}+b_{1}|v_{m,i}^{\Delta}|^{2}) \geq \frac{1}{2^{k}} \right) \leq 2^{k} \mathbb{E}\left(\sum_{|i|> n_k}(b_{2}|u_{m,i}^{\Delta}|^{2}+b_{1}|v_{m,i}^{\Delta}|^{2})\right) < \frac{\eta}{2^{k}}.
\end{equation}
By \eqref{5-62}, for all $m \geq \left[ \frac{T_2}{-\ln\left(1-\frac{\lambda}{4}\Delta\right)}\right]+1$, it follows that
\begin{equation*}
\mathbf{P} \left(\bigcup_{k=1}^\infty \left[\sum_{|i|> n_k}(b_{2}|u_{m,i}^{\Delta}|^{2}+b_{1}|v_{m,i}^{\Delta}|^{2}) \geq \frac{1}{2^{k}} \right]\right) \leq \sum_{k=1}^{\infty}\frac{\eta}{2^{k}} \leq \eta.
\end{equation*}
This implies that for all $m \geq \left[ \frac{T_2}{-\ln\left(1-\frac{\lambda}{4}\Delta\right)}\right]+1$, 
\begin{equation}\label{5-63}
\mathbf{P} \left(\sum_{|i|> n_k}(b_{2}|u_{m,i}^{\Delta}|^{2} + b_{1}|v_{m,i}^{\Delta}|^{2}) \leq \frac{1}{2^{k}}~ for ~all~k\in \mathbb{N} \right) >1- \frac{\eta}{2}. 
\end{equation}
Let $\eta>0$, and define
\begin{equation*}
\begin{aligned}
&\mathbf{Y}_{1,\eta}=\{\psi\in X: \|\psi\|_X\leq R_{1}\},\\
&\mathbf{Y}_{2,\eta}=\{\psi\in X: \sum_{|i|> n_m}(b_{2}|u_{i}|^{2} + b_{1}|v_{i}|^{2}) \leq \frac{1}{2^{k}},~~k\in \mathbb{N}\},\\
&~~~~~~~~~~~~~~\mathbf{Y}_{\eta} = \mathbf{Y}_{1,\eta}\bigcap \mathbf{Y}_{2,\eta}.
\end{aligned}
\end{equation*}
The compactness of $\mathbf{Y}_{\eta}$ can be founded in \cite{Chen-2020}. From \eqref{5-61} and \eqref{5-63}, for any $\eta>0$ and $\Delta\in (0,\Delta^*)$, there exists a compact subset $\mathbf{K}_\eta\in X$, which is independent of $\Delta$. For any $\xi\in X$, there exists $T^*=T^*(\xi)>0$ such that for all integers $m \geq \left[ \frac{T^*}{-\ln\left(1-\frac{\lambda}{4}\Delta\right)}\right]+1$, the following holds:
\begin{equation}\label{5-64}
\mathbf{P} \left(\psi_m^\Delta(\xi)  \in \mathbf{K}_\eta \right) > 1 - \eta.
\end{equation}
Assume $\mu^{\Delta}\in S^{\Delta}$ for some $\Delta\in(0,\Delta^*)$. Because $\mu^{\Delta}$ is an invariant measure, for all $m\in \mathbb{N}$ and any $m_1\in \mathbb{N}$, we have
\begin{equation}\label{5-65}
\int_{X} \mathbf{P} \left( \psi_m^\Delta(\xi)  \in \mathbf{K}_{\eta} \right) \mu^{\Delta} (d\xi) = 
\int_{X} \mathbf{P} \left( \psi_{m+m_1}^\Delta(\xi)  \in \mathbf{K}_{\eta} \right) \mu^{\Delta} (d\xi).
\end{equation}
Additionally,
\begin{equation}\label{5-66}
\int_{X} \mathbf{P} \left( \psi_0^\Delta  \in \mathbf{K}_{\eta} \right) \mu^{\Delta} (d\xi) = 
\int_{X} \mathbf{1}_{\mathbf{K}_{\eta}}\mu^{\Delta} (d\xi) = \mu^{\Delta}(\mathbf{K}_{\eta}).
\end{equation}
From \eqref{5-65}-\eqref{5-66} with $m=0$, for any $m_1\in\mathbb{N}$, we obtain
\begin{equation*}
\mu^{\Delta}(\mathbf{K}_{\eta}) = \int_{X} \mathbf{P} \left( \psi_{m_1}^\Delta(\xi)  \in \mathbf{K}_{\eta} \right) \mu^{\Delta} (d\xi).
\end{equation*}
Using \eqref{5-64} and Fatou's theorem, we then get
\begin{equation}\label{5-67}
\begin{aligned}
\mu^{\Delta}(\mathbf{K}_{\eta}) = &\liminf_{m_{1}\rightarrow\infty}\int_{X} \mathbf{P} \left( \psi_{m_1}^\Delta(\xi)  \in \mathbf{K}_{\eta} \right) \mu^{\Delta} (d\xi)\\
& \geq\int_{X} \liminf_{m_{1}\rightarrow\infty}\mathbf{P} \left(  \psi_{m_1}^\Delta(\xi)  \in \mathbf{K}_{\eta} \right) \mu^{\Delta} (d\xi)\\
& \geq \int_{X} (1-\eta) \mu^{\Delta} (d\xi)\\
& = 1-\eta.
\end{aligned}
\end{equation}
Since $\mu^{\Delta}\in S^{\Delta}$ is arbitrary, this shows that $\cup_{\Delta\in(0,\Delta^{*})}S^\Delta$ is tight, as derived from \eqref{5-67}.

(ii). Suppose that \eqref{5-60} does not hold. Then, there exists a positive constant $\eta$ and a sequence $\Delta_n\rightarrow 0$ such that for all $n\in \mathbb{N}$, we have
\begin{equation}\label{5-68}
d_{P(X)} \left(S^{\Delta_n},S^0\right) \geq \eta.
\end{equation}
From this inequality, we can conclude that there exists a sequence $\{x_{n}\}_{n=1}^{\infty}$ with $x_n\in S^{\Delta_n}$ such that
\begin{equation}\label{5-69}
d_{P(X)} \left(x_{n},S^0\right) \geq \eta, ~~~~n\in \mathbb{N}.
\end{equation}
From (i), the sequence $\{S^{\Delta_{n}}\}_{n=1}^{\infty}$ is compact, so there exists $x_{0}\in P(X)$ such that
\begin{equation}\label{5-70}
\lim_{n\rightarrow \infty} x_{n} = x_{0}.
\end{equation}
By Theorem \ref{theorem3.1} and \cite{Lidingshi-2026}, it follows that $x_0$ belongs to $S^0$. Therefore, from the convergence in \eqref{5-70}, we have
\begin{equation*}
dist\left(x_{n},S^0\right) \leq dist\left(x_{n}, x_{0}\right) \rightarrow 0, ~~~~n\rightarrow \infty,
\end{equation*}
which contradicts \eqref{5-69}. Thus, \eqref{5-60} must be true.
\end{proof}

\section{Finite dimensional approximations}\label{sec:4}

In this section, the finite dimensional approximation for the BEM scheme \eqref{2-1.1} is considered. We provide a $(2N+1)$ dimensional approximation and apply the periodic boundary conditions as in \cite{Han-2020}, that is, $u_{m+1,N}^{\Delta} = u_{m+1,-N-1}^{\Delta},~u_{m+1,-N}^{\Delta} = u_{m+1,N+1}^{\Delta},~v_{m+1,N}^{\Delta} = v_{m+1,-N-1}^{\Delta},~v_{m+1,-N}^{\Delta} = v_{m+1,N+1}^{\Delta}$. Define the linear operator $A_N,~B_N,~B^*_N:X\rightarrow X$ as follows:
\begin{equation*}\label{5-71}
\left(A_Nu\right)_i=
\left\{\begin{array}{l}
         -u_{N-1} + 2u_{N} - u_{-N},~~~i=N\\
         -u_{i-1} + 2u_{i} - u_{i+1},~~~|i|<N\\
         -u_{N} + 2u_{-N} - u_{-N+1},~~~i=-N\\
         0,~~~~|i|>N+1.
 \end{array}\right.
\end{equation*}
\begin{equation*}
\left(B_Nu\right)_i=
\left\{\begin{array}{l}
         u_{-N} - u_{N},~~~i=N\\
         u_{i+1} - u_{i},~~~|i|<N\\
         u_{-N+1} - u_{-N},~~~i=-N\\
         0,~~~~|i|>N+1.
 \end{array}~~~~~~~~~~~~~\right.
\end{equation*}
and
\begin{equation*}
\left(B^*_Nu\right)_i=
\left\{\begin{array}{l}
         u_{N-1} - u_{N},~~~i=N\\
         u_{i-1} - u_{i},~~~|i|<N\\
         u_{N} - u_{-N},~~~i=-N\\
         0,~~~~|i|>N+1.
 \end{array}~~~~~~~~~~~~~~~~~~~\right.
\end{equation*}
The finite dimensional BEM scheme can then be expressed as 
\begin{equation}\label{5-73}
\left\{\begin{array}{l}
         u_{m+1}^{\Delta,N} = u_{m}^{\Delta,N} - d_{1}A_Nu_{m+1}^{\Delta,N}\Delta - a_{1}u_{m+1}^{\Delta,N}\Delta + b_{1}F(u_{m+1}^{\Delta,N},v_{m+1}^{\Delta,N})\Delta
         -b_{2}G(u_{m+1}^{\Delta,N})\Delta + f_N\Delta + \left[h_N+\sigma_N( u_{m}^{\Delta,N})\right]\Delta W_{m},
         \vspace{2.0ex}\\
         v_{m+1}^{\Delta,N} = v_{m}^{\Delta,N} - d_{2}A_Nv_{m+1}^{\Delta,N}\Delta - a_{2}v_{m+1}^{\Delta,N}\Delta - b_{1}F(u_{m+1}^{\Delta,N},v_{m+1}^{\Delta,N})\Delta
         +b_{2}G(u_{m+1}^{\Delta,N})\Delta + g_N\Delta + \left[h_N+\sigma_N( v_{m}^{\Delta,N})\right]\Delta W_{m},
 \end{array}\right.
\end{equation}
with the initial condition
\begin{equation*}
\psi_{0}^{\Delta,N} = \left(u_{0,i},v_{0,i}\right)_{|i|\leq N}\in \mathbb{R}^{2N+1}\times \mathbb{R}^{2N+1},     
\end{equation*}
where
\begin{equation*}
\begin{aligned}
u_{m}^{\Delta,N} = \{u_{m,i}^\Delta\}_{|i|\leq N},~~&v_{m}^{\Delta,N} = \{v_{m,i}^\Delta\}_{|i|\leq N},~~f_{N} = \{f_i \}_{|i|\leq N},  ~~g_{N} = \{g_i \}_{|i|\leq N}, ~~h_{N} = \{h_i \}_{|i|\leq N}, \\ &\sigma_{N}\left(u_{m}^{\Delta,N}\right) = \{\sigma_{i}\left(u_{m,i}^{\Delta}\right)\}_{|i|\leq N}, ~~\sigma_{N}\left(v_{m}^{\Delta,N}\right) = \{\sigma_{i}\left(v_{m,i}^{\Delta}\right)\}_{|i|\leq N}. ~~~~~~~~~~~
\end{aligned}
\end{equation*}
 
For each $u=\left( u_i \right)_{|i|\leq N},~v=\left( v_i \right)_{|i|\leq N}\in \mathbb{R}^{2N+1}$, we can naturally extend these to elements $u=\left( u_i \right)_{i\in \mathbb{Z}},~v=\left( v_i \right)_{i\in \mathbb{Z}}$ in $X$ by defining $u_i=0,~v_i=0$ for all $|i|>N$. 
In the following, we provide an important estimate for the solutions of \eqref{5-73}, denoted by $\psi_{m}^{\Delta,N}$.
\begin{lemma}\label{lemma4.1} Suppose \eqref{5-10}, \eqref{5-11}, \eqref{5-29} hold and $0<\Delta<\Delta^*$. Then, the solutions of \eqref{5-73} satisfy
\begin{equation}
\mathbb{E}\left(\|\psi_{m}^{\Delta,N}\|_X^2\right) \leq \|\psi_0\|_X^2 e^{mln\left(1-\frac{\lambda}{4}\Delta\right)} + M, 
\end{equation}
where $M>0$ is independent of $\Delta$ and $N$.
\end{lemma}
\begin{proof} 
By applying Lemma \ref{lemma2.4}, taking the inner product in $\mathbb{R}^{2N+1}\times \mathbb{R}^{2N+1}$ of \eqref{5-73} with $\left(b_2u_{m+1}^{\Delta,N}, b_1v_{m+1}^{\Delta,N}\right)$, we have
\begin{equation}
\begin{aligned}
& b_{2}\|u_{m+1}^{\Delta,N}\|^2 +  b_{1}\|v_{m+1}^{\Delta,N}\|^2 \\
=& - b_{2}d_{1}\|B_Nu_{m+1}^{\Delta,N}\|^2\Delta - b_{2}a_{1}\|u_{m+1}^{\Delta,N}\|^2\Delta + b_{2}b_{1}(F(u_{m+1}^{\Delta,N},v_{m+1}^{\Delta,N}),u_{m+1}^{\Delta,N})\Delta\vspace{1.0ex}\\
& - b_{2}^2(G(u_{m+1}^{\Delta,N}),u_{m+1}^{\Delta,N})\Delta + b_{2}(f_N,u_{m+1}^{\Delta,N})\Delta + b_{2}\left(u_{m}^{\Delta,N} +\left[h_N+\sigma_N( u_{m}^{\Delta,N})\right]\Delta W_{m},u_{m+1}^{\Delta,N}\right)\\
& - b_{1}d_{2}\|B_Nv_{m+1}^{\Delta,N}\|^2\Delta - b_{1}a_{2}\|v_{m+1}^{\Delta,N}\|^2\Delta - b_{1}^2(F(u_{m+1}^{\Delta,N},v_{m+1}^{\Delta,N}))\Delta\vspace{1.0ex}\\
& + b_{1}b_{2}(G(u_{m+1}^{\Delta,N}),v_{m+1}^{\Delta,N})\Delta + b_{1}(g_N,v_{m+1}^{\Delta,N})\Delta + b_{1}\left(v_{m}^{\Delta,N} + \left[h_N+\sigma_N( v_{m}^{\Delta,N})\right]\Delta W_{m},v_{m+1}^{\Delta,N})\right)\\
\leq & (\frac{1}{2}-\frac{\lambda}{2}\Delta)b_{2}\|u_{m+1}^{\Delta,N}\|^2 + b_{2}\Delta\frac{1}{2\lambda}\|f_N\|^2 + \frac{1}{2}b_{2}\|u_{m}^{\Delta,N}\|^2\\
& + \frac{1}{2}b_{2}\|\left(h_N+\sigma_N( u_{m}^{\Delta,N})\right)\Delta W_{m}\|^2 + b_{2}\left(u_{m}^{\Delta,N}, \left(h_N+\sigma_N( u_{m}^{\Delta,N})\right)\Delta W_{m}\right)\\
& + (\frac{1}{2}-\frac{\lambda}{2}\Delta)b_{1}\|v_{m+1}^{\Delta,N}\|^2 + b_{1}\Delta\frac{1}{2\lambda}\|g_N\|^2 + \frac{1}{2}b_{1}\|v_{m}^{\Delta,N}\|^2\\
& +\frac{1}{2}b_{1}\|\left(h_N+\sigma_N( v_{m}^{\Delta,N})\right)\Delta W_{m}\|^2 + b_{1}\left(v_{m}^{\Delta,N}, \left(h_N+\sigma_N( v_{m}^{\Delta,N})\right)\Delta W_{m}\right).
\end{aligned}
\end{equation}
Then, we get
\begin{equation}
\begin{aligned}
&\mathbb{E}\left(\|\psi_{m+1}^{\Delta,N}\|_X^2\right) = \mathbb{E}\left(b_{2}\|u_{m+1}^{\Delta,N}\|^2 +  b_{1}\|v_{m+1}^{\Delta,N}\|^2\right) \\
\leq & \left(1-\frac{\lambda}{4}\Delta\right)^{m+1}\left(b_2\|u_{0}^{\Delta,N}\|^2 + b_1\|v_{0}^{\Delta,N}\|^2\right) \\
& + \Delta\left(\frac{b_{2}}{\lambda}\|f_N\|^2 + \frac{b_{1}}{\lambda}\|g_N\|^2 + 2b_2\|h_N\|^2 +2b_1\|h_N\|^2 +4b_2\|\delta_N\|^2 +4b_1\|\delta_N\|^2 \right) \sum^m_{l=0}\left(1-\frac{\lambda}{4}\Delta\right)^{l}\\
\leq & \left(1-\frac{\lambda}{4}\Delta\right)^{m+1}\|\psi_0^{\Delta,N}\|_X^2 + \left(\frac{b_{2}}{\lambda}\|f_N\|^2 + \frac{b_{1}}{\lambda}\|g_N\|^2 + 2b_2\|h_N\|^2 +2b_1\|h_N\|^2 +4b_2\|\delta_N\|^2 +4b_1\|\delta_N\|^2 \right)\frac{4}{\lambda}. 
\end{aligned}
\end{equation}
That is, 
\begin{equation}
\mathbb{E}\left(\|\psi_{m}^{\Delta,N}\|_X^2\right) \leq \|\psi_0^{\Delta,N}\|_X^2 e^{mln\left(1-\frac{\lambda}{4}\Delta\right)} + M. 
\end{equation}
\end{proof}

The proof of the following lemma follows similarly to that of Lemma \ref{lemma2.5}, therefore we omit it here.

\begin{lemma}\label{lemma4.2} Suppose \eqref{5-10}, \eqref{5-11}, \eqref{5-29} hold and $0<\Delta<\Delta^*$. Then, for any $\eta>0$ and a bounded set $B\in \mathbb{R}^{2N+1}\times \mathbb{R}^{2N+1}$, there exists $I=I(\eta,B)\in \mathbb{N}$, independent of $N$, such that for $N\in \mathbb{N}$ and $\psi_{0}\in B$, the solutions of \eqref{5-73} satisfy
\begin{equation*}
\mathbb{E}\left(\sum_{I < |i| \leq N}(b_{2}|u_{m,i}^{\Delta}|^{2}+b_{1}|v_{m,i}^{\Delta}|^{2})\right) 
\leq \sum_{|i|\leq N}(b_{2}|u_{0,i}^{\Delta}|^{2}+b_{1}|v_{0,i}^{\Delta}|^{2}) e^{mln\left(1-\frac{\lambda}{4}\Delta\right)} + \eta. 
\end{equation*}
\end{lemma}

Using Lemma \ref{lemma4.1}, Lemma \ref{lemma4.2}, and the classical Krylov-Bogolyubov method, we establish the existence of invariant measures for the finite dimensional truncated system \eqref{5-73} in $\mathbb{R}^{2N+1}\times \mathbb{R}^{2N+1}$.

\begin{theorem}\label{theorem4.1} Suppose \eqref{5-10}, \eqref{5-11} and \eqref{5-29} hold. Then, there exists a numerical invariant measure for the finite dimensional truncated system \eqref{5-73} in $R^{2N+1}\times R^{2N+1}$.
\end{theorem}

To establish the convergence of the numerical invariant measures for the finite dimensional truncated system, we first present the following conclusion.

\begin{lemma}\label{lemma4.3} Suppose \eqref{5-10}, \eqref{5-11} hold and $0<\Delta<\Delta^{*}$. Then, for each compact set $K\in X$, $m\in \mathbb{N}$ and $\eta>0$, the following holds:
\begin{equation}
\lim_{N\rightarrow\infty}\sup_{\psi_0\in K\bigcap R^{2N+1}\times R^{2N+1}}  \mathbf{P}\left(\| \psi_m^{\Delta}(\psi_0) - \psi_m^{\Delta,N}(\psi_0)\|_X\geq \eta \right) = 0. 
\end{equation}
\end{lemma}
\begin{proof}
Denote
\begin{equation*}
\sigma_{N^{*}}(0) = \left(\sigma_{i}^{*}(0)\right)_{i\in \mathbb{Z}} = \left\{\begin{array}{l}
         0,~~~~~~~~~~|i|\leq N,
         \vspace{2.0ex}\\
         \sigma_{i}(0),~~~~|i|> N.
 \end{array}\right.
\end{equation*}
From \eqref{2-1.1} and \eqref{5-73}, it follows that
\begin{equation}\label{5-74}
\left\{\begin{array}{l}
         u_{m+1}^{\Delta} - u_{m+1}^{\Delta,N} =  u_{m}^{\Delta} - u_{m}^{\Delta,N} - d_{1}\Delta \left(Au_{m+1}^{\Delta} - A_Nu_{m+1}^{\Delta,N}\right) - a_{1}\Delta \left(u_{m+1}^{\Delta} - u_{m+1}^{\Delta,N}\right) \\
 \hspace{13.0ex} + b_{1}\Delta \left(F(u_{m+1}^{\Delta},v_{m+1}^{\Delta}) - F(u_{m+1}^{\Delta,N},v_{m+1}^{\Delta,N}) \right) - b_{2}\Delta \left(G(u_{m+1}^{\Delta}) - G(u_{m+1}^{\Delta,N})\right) \\
 \hspace{13.0ex} + \left( f - f_N \right) \Delta + \left[(h - h_N) + \left( \sigma( u_{m}^{\Delta}) - \sigma_N( u_{m}^{\Delta,N})\right)\right]\Delta W_{m},
         \vspace{2.0ex}\\
         v_{m+1}^{\Delta} - v_{m+1}^{\Delta,N} = v_{m}^{\Delta} - v_{m}^{\Delta,N} - d_{2}\Delta \left(Av_{m+1}^{\Delta} - A_Nv_{m+1}^{\Delta,N}\right) - a_{2}\Delta \left(v_{m+1}^{\Delta} - v_{m+1}^{\Delta,N}\right)\\
 \hspace{13.0ex} - b_{1}\Delta \left(F(u_{m+1}^{\Delta},v_{m+1}^{\Delta}) - F(u_{m+1}^{\Delta,N},v_{m+1}^{\Delta,N}) \right) + b_{2}\Delta \left(G(u_{m+1}^{\Delta}) - G(u_{m+1}^{\Delta,N})\right) \\
\hspace{13.0ex} + \left( g - g_N \right) \Delta  + \left[(h - h_N) + \left( \sigma( v_{m}^{\Delta}) - \sigma_N( v_{m}^{\Delta,N})\right)\right]\Delta W_{m}.
 \end{array}\right.
\end{equation}
Taking the inner product of \eqref{5-74} with $\left(b_2  \left(u_{m+1}^{\Delta} - u_{m+1}^{\Delta,N}\right),  b_1\left(v_{m+1}^{\Delta} - v_{m+1}^{\Delta,N}\right)\right)$ in $X$, we obtain
\begin{equation}\label{5-75}
\begin{aligned}
b_{2} \|u_{m+1}^{\Delta} - u_{m+1}^{\Delta,N}\|^2 = & b_{2} \left(u_{m}^{\Delta} - u_{m}^{\Delta,N}, u_{m+1}^{\Delta} - u_{m+1}^{\Delta,N}\right) - b_{2}d_{1}\Delta \left(Au_{m+1}^{\Delta} - A_{N}u_{m+1}^{\Delta,N}, u_{m+1}^{\Delta} - u_{m+1}^{\Delta,N}\right)\\
& - b_{2}a_{1}\Delta \|u_{m+1}^{\Delta} - u_{m+1}^{\Delta,N}\|^2 + b_{2}b_{1}\Delta \left(F(u_{m+1}^{\Delta},v_{m+1}^{\Delta}) - F(u_{m+1}^{\Delta,N},v_{m+1}^{\Delta,N}), u_{m+1}^{\Delta} - u_{m+1}^{\Delta,N}\right)\\
& - b_{2}^2\Delta \left(G(u_{m+1}^{\Delta}) - G(u_{m+1}^{\Delta,N}), u_{m+1}^{\Delta} - u_{m+1}^{\Delta,N}\right)  + b_{2}\Delta \left(f - f_N, u_{m+1}^{\Delta} - u_{m+1}^{\Delta,N} \right) \\
& + b_{2}\left(h - h_N, u_{m+1}^{\Delta} - u_{m+1}^{\Delta,N}\right)\Delta W_{m} + b_{2}\left( \sigma( u_{m}^{\Delta}) - \sigma_N( u_{m}^{\Delta,N}), u_{m+1}^{\Delta} - u_{m+1}^{\Delta,N}\right)\Delta W_{m}\\
\leq & b_{2} \frac{\lambda\Delta}{8} \|u_{m+1}^{\Delta} - u_{m+1}^{\Delta,N}\|^2 + b_{2}\frac{2}{\lambda\Delta}\|u_{m}^{\Delta} - u_{m}^{\Delta,N}\|^2 - b_{2}d_{1}\Delta \left(Au_{m+1}^{\Delta} - A_{N}u_{m+1}^{\Delta,N}, u_{m+1}^{\Delta} - u_{m+1}^{\Delta,N}\right)\\
& - b_{2}\lambda\Delta \|u_{m+1}^{\Delta} - u_{m+1}^{\Delta,N}\|^2 + b_{2}b_{1}\Delta \left(F(u_{m+1}^{\Delta},v_{m+1}^{\Delta}) - F(u_{m+1}^{\Delta,N},v_{m+1}^{\Delta,N}), u_{m+1}^{\Delta} - u_{m+1}^{\Delta,N}\right)\\
& - b_{2}^2\Delta \left(G(u_{m+1}^{\Delta}) - G(u_{m+1}^{\Delta,N}), u_{m+1}^{\Delta} - u_{m+1}^{\Delta,N}\right)  + b_{2}\Delta\frac{\lambda}{8} \|u_{m+1}^{\Delta} - u_{m+1}^{\Delta,N}\|^2 + b_{2}\Delta\frac{2}{\lambda} \|f - f_N\|^2 \\
& + b_{2}\left(h - h_N, u_{m+1}^{\Delta} - u_{m+1}^{\Delta,N}\right)\Delta W_{m} + b_{2}\left( \sigma( u_{m}^{\Delta}) - \sigma( u_{m}^{\Delta,N}), u_{m+1}^{\Delta} - u_{m+1}^{\Delta,N}\right)\Delta W_{m} \\
& + b_{2}\left(\sigma_{N^*}(0), u_{m+1}^{\Delta} - u_{m+1}^{\Delta,N}\right)\Delta W_{m}.
\end{aligned}
\end{equation}
With the same method, we get
\begin{equation}\label{5-75-1}
\begin{aligned}
b_{1} \|v_{m+1}^{\Delta} - v_{m+1}^{\Delta,N}\|^2 \leq & b_{1} \frac{\lambda\Delta}{8} \|v_{m+1}^{\Delta} - v_{m+1}^{\Delta,N}\|^2 + b_{1}\frac{2}{\lambda\Delta}\|v_{m}^{\Delta} - v_{m}^{\Delta,N}\|^2 - b_{1}d_{2}\Delta \left(Av_{m+1}^{\Delta} - A_{N}v_{m+1}^{\Delta,N}, v_{m+1}^{\Delta} - v_{m+1}^{\Delta,N}\right)\\
& - b_{1}\lambda\Delta \|v_{m+1}^{\Delta} - v_{m+1}^{\Delta,N}\|^2 - b_{1}^2\Delta \left(F(u_{m+1}^{\Delta},v_{m+1}^{\Delta}) - F(u_{m+1}^{\Delta,N},v_{m+1}^{\Delta,N}), v_{m+1}^{\Delta} - v_{m+1}^{\Delta,N}\right)\\
& + b_{1}b_{2}\Delta \left(G(u_{m+1}^{\Delta}) - G(u_{m+1}^{\Delta,N}), v_{m+1}^{\Delta} - v_{m+1}^{\Delta,N}\right)  + b_{1}\Delta\frac{\lambda}{8} \|v_{m+1}^{\Delta} - v_{m+1}^{\Delta,N}\|^2 + b_{1}\Delta\frac{2}{\lambda} \|g - g_N\|^2 \\
& + b_{1}\left(h - h_N, v_{m+1}^{\Delta} - v_{m+1}^{\Delta,N}\right)\Delta W_{m} + b_{1}\left( \sigma( v_{m}^{\Delta}) - \sigma( v_{m}^{\Delta,N}), v_{m+1}^{\Delta} - v_{m+1}^{\Delta,N}\right)\Delta W_{m} \\
& + b_{1}\left(\sigma_{N^*}(0), v_{m+1}^{\Delta} - v_{m+1}^{\Delta,N}\right)\Delta W_{m}.
\end{aligned}
\end{equation}
For the third term on the right hand side of $\eqref{5-75}$, by Young's inequality, we have
\begin{equation}\label{5-76}
\begin{aligned}
& - b_{2}d_{1}\Delta \left(Au_{m+1}^{\Delta} - A_Nu_{m+1}^{\Delta,N}, u_{m+1}^{\Delta} - u_{m+1}^{\Delta,N}\right) \\
= & - b_{2}d_{1}\Delta \left(Au_{m+1}^{\Delta} - A_{N}u_{m+1}^{\Delta}, u_{m+1}^{\Delta} - u_{m+1}^{\Delta,N}\right) - b_{2}d_{1}\Delta \left(A_{N}u_{m+1}^{\Delta} - A_{N}u_{m+1}^{\Delta,N}, u_{m+1}^{\Delta} - u_{m+1}^{\Delta,N}\right) \\
\leq & b_{2} d_{1}^{2}\Delta\frac{2}{\lambda} \|Au_{m+1}^{\Delta} - A_{N}u_{m+1}^{\Delta}\|^2 + b_{2}\Delta\frac{\lambda}{8}\|u_{m+1}^{\Delta} - u_{m+1}^{\Delta,N}\|^2 \\
= & b_{2} d_{1}^{2}\Delta\frac{2}{\lambda}  \sum_{|i|>N}\left( (Au_{m+1}^{\Delta})_{i} \right)^{2} +  b_{2}\Delta\frac{\lambda}{8}\|u_{m+1}^{\Delta} - u_{m+1}^{\Delta,N}\|^2.
\end{aligned}
\end{equation}
Using the same method, we obtain
\begin{equation}\label{5-76-1}
- b_{1}d_{2}\Delta \left(Av_{m+1}^{\Delta} - A_Nv_{m+1}^{\Delta,N}, v_{m+1}^{\Delta} - v_{m+1}^{\Delta,N}\right) 
\leq  b_{1} d_{2}^{2}\Delta\frac{2}{\lambda}  \sum_{|i|>N}\left( (Av_{m+1}^{\Delta})_{i} \right)^{2} +  b_{1}\Delta\frac{\lambda}{8}\|v_{m+1}^{\Delta} - v_{m+1}^{\Delta,N}\|^2.
\end{equation}
For the last three terms on the right hand side of $\eqref{5-75}$, by $\eqref{5-10}$, we have
\begin{equation}\label{5-78}
\begin{aligned}
& b_{2}\left(h - h^N, u_{m+1}^{\Delta} - u_{m+1}^{\Delta,N}\right)\Delta W_{m} + b_{2}\left( \sigma( u_{m}^{\Delta}) - \sigma( u_{m}^{\Delta,N}), u_{m+1}^{\Delta} - u_{m+1}^{\Delta,N}\right)\Delta W_{m} \\
& ~~~ + b_{2}\left(\sigma_{N^*}(0), u_{m+1}^{\Delta} - u_{m+1}^{\Delta,N}\right)\Delta W_{m}\\
& \leq   b_{2}\frac{2}{\lambda\Delta} \|\sigma( u_{m}^{\Delta}) - \sigma( u_{m}^{\Delta,N})\|^2 |\Delta W_{m}|^2 + \frac{3}{8}b_{2}\lambda\Delta \|u_{m+1}^{\Delta} - u_{m+1}^{\Delta,N}\|^2 \\
& ~~~ + b_{2}\frac{4}{\lambda\Delta} \sum_{|i|>N}|\delta_{i}|^2|\Delta W_{m}|^2 + b_{2}\frac{2}{\lambda\Delta} \sum_{|i|>N}|h_{i}|^2|\Delta W_{m}|^2 \\
& \leq  b_{2}\frac{2}{\lambda\Delta} L_{\sigma}\| u_{m}^{\Delta} - u_{m}^{\Delta,N}\|^2 |\Delta W_{m}|^2 + \frac{3}{8}b_{2}\lambda\Delta \|u_{m+1}^{\Delta} - u_{m+1}^{\Delta,N}\|^2 \\
& ~~~ + b_{2}\frac{4}{\lambda\Delta} \sum_{|i|>N}|\delta_{i}|^2|\Delta W_{m}|^2 + b_{2}\frac{2}{\lambda\Delta} \sum_{|i|>N}|h_{i}|^2|\Delta W_{m}|^2.
\end{aligned}
\end{equation}
Using the same method, we obtain
\begin{equation}\label{5-78-1}
\begin{aligned}
& b_{1}\left(h - h^N, v_{m+1}^{\Delta} - v_{m+1}^{\Delta,N}\right)\Delta W_{m} + b_{1}\left( \sigma( v_{m}^{\Delta}) - \sigma( v_{m}^{\Delta,N}), v_{m+1}^{\Delta} - v_{m+1}^{\Delta,N}\right)\Delta W_{m} \\
& ~~~ + b_{1}\left(\sigma_{N^*}(0), v_{m+1}^{\Delta} - v_{m+1}^{\Delta,N}\right)\Delta W_{m} \\
& \leq  b_{1}\frac{2}{\lambda\Delta} L_{\sigma}\| v_{m}^{\Delta} - v_{m}^{\Delta,N}\|^2 |\Delta W_{m}|^2 + \frac{3}{8}b_{1}\lambda\Delta \|v_{m+1}^{\Delta} - v_{m+1}^{\Delta,N}\|^2 \\
& ~~~ + b_{1}\frac{4}{\lambda\Delta} \sum_{|i|>N}|\delta_{i}|^2|\Delta W_{m}|^2 + b_{1}\frac{2}{\lambda\Delta} \sum_{|i|>N}|h_{i}|^2|\Delta W_{m}|^2.
\end{aligned}
\end{equation}
Combining $\eqref{5-75}$-$\eqref{5-78-1}$ and taking the expectation, we obtain
\begin{equation}\label{5-79}
\begin{aligned}
& \mathbb{E}\left(b_2\|u_{m+1}^{\Delta} - u_{m+1}^{\Delta,N}\|^2  + b_1\|v_{m+1}^{\Delta} - v_{m+1}^{\Delta,N}\|^2 \right)\\
\leq & \left(1 + \frac{\lambda\Delta}{4} \right)^{-1}\left( \frac{2}{\lambda\Delta} + \frac{2}{\lambda}L_{\sigma} \right) \left( b_{2}\|u_{m}^{\Delta} - u_{m}^{\Delta,N}\|^2 + b_{1}\|v_{m}^{\Delta} - v_{m}^{\Delta,N}\|^2 \right) \\
& + \frac{2}{\lambda}\Delta\left(1 + \frac{\lambda\Delta}{4} \right)^{-1}\left[\mathbb{E} \left( b_{2} d_{1}^{2} \sum_{|i|>N}\left( (Au_{m+1}^{\Delta})_{i} \right)^{2} + b_{1} d_{2}^{2} \sum_{|i|>N}\left( (Av_{m+1}^{\Delta})_{i} \right)^{2}  \right) \right]\\
&+ \frac{2}{\lambda} \left( 1 + \frac{\lambda\Delta}{4} \right)^{-1} \left[ b_{2}\Delta \sum_{|i|>N} |f_i|^2  + b_{1}\Delta  \sum_{|i|>N} |g_i|^2 + 2(b_1+b_2) \sum_{|i|>N}|\delta_{i}|^2 +  (b_1 + b_2)\sum_{|i|>N}|h_{i}|^2 \right].
\end{aligned}
\end{equation}
By applying Lemma \ref{lemma2.5}, for any $\eta>0$ and compact set $K\subseteq X$, there exists $I_1=I_1(\eta, K)>0$ such that for any $N>I_1$, 
\begin{equation}\label{5-80}
\frac{2}{\lambda}\Delta\left(1 + \frac{\lambda\Delta}{4} \right)^{-1}\left[\mathbb{E} \left( b_{2} d_{1}^{2} \sum_{|i|>N}\left( (Au_{m+1}^{\Delta})_{i} \right)^{2} + b_{1} d_{2}^{2} \sum_{|i|>N}\left( (Av_{m+1}^{\Delta})_{i} \right)^{2}  \right) \right] < \frac{\eta}{2}.
\end{equation}
Since $f, g, h,\delta\in \ell^2$, there exists $I_2=I(\eta) > I_1$ such that for any $N>I_2$, the following holds:
\begin{equation}\label{5-81}
\frac{2}{\lambda} \left( 1 + \frac{\lambda\Delta}{4} \right)^{-1} \left[ b_{2}\Delta \sum_{|i|>N} |f_i|^2  + b_{1}\Delta \sum_{|i|>N} |g_i|^2 + 2(b_1+b_2) \sum_{|i|>N}\delta_{i}^2 +  (b_1 + b_2)\sum_{|i|>N}h_{i}^2 \right] < \frac{\eta}{2}.
\end{equation}
Moreover, there exists $I_3=I_3(\eta, K) > 0$ such that for any $N>I_3$ and $\psi_0=(u_0, v_0)^T\in K$, 
\begin{equation}\label{5-82}
\left( \sum_{|i|>N} b_2|u_{0,i}|^2 + b_1|v_{0,i}|^2 \right) < \eta. 
\end{equation}
By substituting $\eqref{5-80}$-$\eqref{5-82}$ into $\eqref{5-79}$, for $N>I_3$ and $\psi_0=(u_0, v_0)^T\in K$, we have
\begin{equation*}
\begin{aligned}
& \mathbb{E}\left(b_2\|u_{m+1}^{\Delta} - u_{m+1}^{\Delta,N}\|^2  + b_1\|v_{m+1}^{\Delta} - v_{m+1}^{\Delta,N}\|^2 \right)\\
\leq & \eta \left[ \left(1 + \frac{\lambda\Delta}{4} \right)^{-1}\left( \frac{2}{\lambda\Delta} + \frac{2}{\lambda}L_{\sigma} \right) \right]^{m+1} + \eta \sum_{l=0}^{m} \left[ \left(1 + \frac{\lambda\Delta}{4} \right)^{-1}\left( \frac{2}{\lambda\Delta} + \frac{2}{\lambda}L_{\sigma} \right) \right]^{l}\\
\leq & \eta C.
\end{aligned}
\end{equation*}
By applying Markov's inequality, the lemma is proven.
\end{proof}

Let $S^{\Delta,N}$ denote the set of all invariant measures associated with $\psi_m^{\Delta,N}$. According to Theorem \ref{theorem4.1}, for each $\Delta\in (0,\Delta^*)$ and $N\in \mathbb{N}$, the set $S^{\Delta,N}$ is nonempty under the conditions \eqref{5-10}, \eqref{5-11} and \eqref{5-29}.
  
By applying Lemma \ref{lemma4.1} and Lemma \ref{lemma4.2}, and using a proof technique similar to the first part of Theorem \ref{theorem3.2}, we have the following result.
 
\begin{lemma}\label{lemma4.4} Suppose \eqref{5-10}, \eqref{5-11} and \eqref{5-29} hold. Then the set $\bigcup_{N\in \mathbb{N}}S^{\Delta, N}$ is compact in $\left( P(X),~ d_{P(X)} \right)$.
\end{lemma}

\begin{lemma}\label{lemma4.5} Suppose \eqref{5-10}, \eqref{5-11}, \eqref{5-29} hold and $N_n\rightarrow \infty$. If $\mu$ is a probability measure on $X$, and $\mu^{\Delta,N_n}$ is an invariant measure of $\psi_{m}^{\Delta,N_n}$ which weakly converges to $\mu$, then $\mu$ must be an invariant measure of $\psi^\Delta_{m}$.
\end{lemma}
\begin{proof}
Let $UC_{b}(X)$ denote the Banach space of all bounded uniformly continuous functions on $X$. Here, we only need to prove that for any $\varphi\in U C_b(X)$ and $m\in \mathbb{N}$, the following holds:
\begin{equation}\label{5-83}
\int_{X} \mathbb{E}\varphi\left( \psi_m^\Delta(x)\right)\mu (dx) = \int_{X} \varphi(x)\mu (dx).
\end{equation}
From Lemma \ref{lemma4.4}, we know that $\{\mu^{\Delta,N_n}\}$ is compact. Therefore, for any $\eta >0$, there exists a compact set $K = K(\eta)\subseteq X$ such that
\begin{equation}\label{5-84}
\mu^{\Delta,N_n}(K) \geq 1 - \eta, n\in\mathbb{N}^{+}.
\end{equation}
By \eqref{5-84}, we have
\begin{equation}\label{5-85}
\begin{aligned}
&|\int_{X} \mathbb{E}\varphi\left( \psi_m^\Delta(x)\right)\mu^{\Delta,N_{n}} (dx) - \int_{X} \varphi(x)\mu^{\Delta,N_{n}} (dx)|  \\
\leq & | \int_{K} \mathbb{E} \varphi\left( \psi_m^\Delta(x) \right)\mu^{\Delta,N_{n}}(dx) -  \int_{K} \varphi(x)\mu^{\Delta,N_{n}} (dx)| + 2\eta\sup_{x\in X}|\varphi(x)| \\
\leq & | \int_{K\bigcap R^{2N_n+1}\times R^{2N_n+1}} \mathbb{E} \varphi\left( \psi_m^\Delta(x) \right)\mu^{\Delta,N_{n}}(dx) -  \int_{K\bigcap R^{2N_n+1}\times R^{2N_n+1}} \varphi(x)\mu^{\Delta,N_{n}} (dx)| + 2\eta\sup_{x\in X}|\varphi(x)| \\
= &| \int_{K\bigcap R^{2N_n+1}\times R^{2N_n+1}} \mathbb{E} \varphi\left( \psi_m^\Delta(x) \right)\mu^{\Delta,N_{n}}(dx) -  \int_{K\bigcap R^{2N_n+1}\times R^{2N_n+1}} \mathbb{E} \varphi\left( \psi_m^{\Delta,N_n}(x) \right)\mu^{\Delta,N_{n}} (dx)| + 2\eta\sup_{x\in X}|\varphi(x)|\\
\leq & \int_{K\bigcap R^{2N_n+1}\times R^{2N_n+1}} \mathbb{E}|\varphi\left( \psi_m^\Delta(x) - \psi_m^{\Delta,N_n}(x)\right)|\mu^{\Delta,N_{n}} (dx) + 2\eta\sup_{x\in X}|\varphi(x)|.
\end{aligned}
\end{equation}
Since $\varphi\in U C_b(\mathrm{X})$, for any $\eta >0$, there exists $\epsilon>0$ such that for all $y,z\in X$ with $\|y-z\|<\epsilon$, we have $|\varphi(y)-\varphi(z)|<\eta$. Therefore, it follows that
\begin{equation}\label{5-86}
\begin{aligned}
& \int_{K\bigcap R^{2N_n+1}\times R^{2N_n+1}} \mathbb{E}|\varphi\left( \psi_m^\Delta(x)\right) - \varphi\left( \psi_m^{\Delta,N_n}(x)\right)|\mu^{\Delta,N_{n}} (dx) \\
= & \int_{K \bigcap R^{2N_n+1}\times R^{2N_n+1}} \left(\int_{\|\psi_m^\Delta(x) - \psi_m^{\Delta,N_n}(x)\| \geq \epsilon} |\varphi \left( \psi_m^\Delta(x)\right) - \varphi\left( \psi_m^{\Delta,N_n}(x)\right)| \mathbf{P}(d\omega)\right)\mu^{\Delta,N_{n}} (dx) \\
& + \int_{K \bigcap R^{2N_n+1}\times R^{2N_n+1}} \left(\int_{\|\psi_m^\Delta(x) - \psi_m^{\Delta,N_n}(x)\| < \epsilon} |\varphi \left( \psi_m^\Delta(x)\right) - \varphi\left( \psi_m^{\Delta,N_n}(x)\right)| \mathbf{P}(d\omega)\right)\mu^{\Delta,N_{n}} (dx)\\
\leq & 2\sup_{x\in X}|\varphi(x)| \sup_{x\in K \bigcap R^{2N_n+1}\times R^{2N_n+1} }\mathbf{P}\left( \|  \psi_m^\Delta(x) - \psi_m^{\Delta,N_n}(x)  \|    \geq\eta\right) + \eta,
\end{aligned}
\end{equation}
By Lemma \ref{lemma4.3} and \eqref{5-85}-\eqref{5-86}, we have 
\begin{equation}\label{5-87}
\lim_{n\rightarrow \infty} |\int_{X} \mathbb{E}\varphi\left( \psi_m^\Delta(x)\right)\mu^{\Delta,N_{n}} (dx) - \int_{X} \varphi(x)\mu^{\Delta,N_{n}} (dx)| \leq \eta +2\eta\sup_{x\in X}|\varphi(x)|.
\end{equation}
Since $\eta >0 $ is arbitrary and $\mu^{\Delta,N_n}$ weakly converges to $\mu$, then \eqref{5-83} follows from \eqref{5-87}, and $\mu$ is an invariant measure of $\psi_m^\Delta$.
\end{proof}

By applying Lemma \ref{lemma4.4} and Lemma \ref{lemma4.5}, and similar to the proof method of Theorem \ref{theorem3.2}, the following result is established.
\begin{theorem}\label{theorem4.2} Suppose \eqref{5-10}, \eqref{5-11} and \eqref{5-29} hold. Then
\begin{equation*}
\lim_{N\rightarrow \infty} d_{P(X)} \left(S^{\Delta,N},S^{\Delta}\right) = 0.
\end{equation*}
\end{theorem}

Based on Theorem \ref{theorem3.2} and Theorem \ref{theorem4.2}, the following result describes the convergence of the numerical invariant measure set $S^{\Delta,N}$ to the invariant measures set $S^{0}$ of the random lattice system.

\begin{theorem}\label{theorem4.3} Suppose \eqref{5-10}, \eqref{5-11} and \eqref{5-29} hold. Then
\begin{equation*}
\lim_{\Delta\rightarrow 0}\lim_{N\rightarrow \infty} d_{P(X)} \left(S^{\Delta,N},S^{0}\right) = 0.
\end{equation*}
\end{theorem}


\end{document}